\documentclass[12pt,a4paper,twoside,reqno]{amsart}

\usepackage{tikz}
\usepackage{slashed}
\usepackage{amsmath,amscd}
\usepackage{amssymb,amsfonts,mathrsfs}
\usepackage[driver=pdftex,margin=3cm,heightrounded=true,centering]{geometry}
\usepackage{JK}
\usepackage[colorlinks=true,linkcolor=blue,citecolor=blue]{hyperref}

\newtheorem{thm}{Theorem}[section]
\newtheorem{notation}[thm]{Notation}
\usepackage{graphicx}

\begin{document}
\title[Anomalies of Dirac type operators on Euclidean space]{Anomalies of Dirac type operators on Euclidean space}

\author{Alan Carey}
\address{Mathematical Sciences Institute, Australian National University, Canberra, ACT, Australia}
\email{alan.carey@anu,edu,au}

\author{Harald Grosse}
\address{Department of Physics, University of Vienna, Boltzmanngasse, Vienna}
\email{harald.grosse@univie.ac.at}

\author{Jens Kaad}
\address{International School of Advanced Studies (SISSA),
Via Bonomea 265,
34136 Trieste,
Italy}
\email{jenskaad@hotmail.com}

%
%
%
\thanks{All authors thank the referees for comments that have improved the exposition.
The first author thanks the Alexander von Humboldt Stiftung and colleagues at the University of M\"unster and acknowledges the support of the Australian Research Council. The third author is supported by the Fondation Sciences Math\'ematiques de Paris (FSMP) and by a public grant overseen by the French National Research Agency (ANR) as part of the ``Investissements d'Avenir'' program (reference: ANR-10-LABX-0098). All authors are very appreciative of the support offered by the Erwin Schr\"odinger Institute where much of this research was carried out. We are also grateful for the advice and wisdom of Fritz Gesztesy while this investigation was proceeding.}

\begin{abstract} We develop by example a type of index theory for non-Fredholm operators. A general framework using cyclic homology for this
notion of index was introduced in a separate article \cite{CaKa:TIH} where it may be seen to generalise earlier ideas of Carey-Pincus and Gesztesy-Simon on this problem.
Motivated by an example in two dimensions in \cite{BGGSS:WKS} we introduce in this paper a class of examples of Dirac type operators on $\rr^{2n}$ 
that provide non-trivial examples of our homological approach. Our examples may be seen as extending old ideas about the notion of anomaly 
introduced by physicists to handle topological terms in quantum action principles with an important difference, namely we are dealing with purely geometric
data that can be seen to arise from the continuous spectrum of our Dirac type operators.
\end{abstract}

\maketitle

\section{Introduction}
\subsection{Background}
In two interesting papers from the latter part of the previous century R. W. Carey and J. Pincus in \cite{CaPi:IOG}, and F. Gesztesy and B. Simon in \cite{GeSi:TIW} made a start on an `index
theory' for non-Fredholm operators. These two papers study different aspects of the problem using related techniques.  In both papers this `index' is expressed in terms of the Krein spectral shift function from scattering theory.\footnote{It is termed the Witten index in \cite{GeSi:TIW}
but in fact Witten considered a scaling limit of a generalised McKean-Singer formula, that is, an
expression in terms of heat semigroups, \cite{Wi:CSB}. It was Gesztesy-Simon who discovered the connection between Witten's idea and
the spectral shift function and hence to Carey-Pincus.}  A comprehensive list of papers on the Witten index may be found in \cite{CGPST:WSS}.

In a companion paper \cite{CaKa:TIH} we developed a formalism based on cyclic homology in which this previous work can be seen to fit as a special case.
We termed the spectral invariant constructed in \cite{CaKa:TIH} the `homological index'.
The Gesztesy-Simon index is related to a scaling limit of the homological index
which is a functional defined on the homology of a certain bicomplex constructed using techniques from 
cyclic theory \cite{Lod:CH}. 
However it is not at all clear whether there are interesting examples of this formalism except in the very simplest case studied in \cite{BGGSS:WKS}.

In noncommutative geometry, from the spectral point of view, we start with an unbounded selfadjoint operator $\C D$ densely defined on a Hilbert space $\C H$. Then we introduce a subalgebra of the algebra of bounded operators $\sL(\C H)$ on $\C H$ and probe the structure of this algebra using Kasparov theory (computing in particular the topological index as a pairing in $K$-theory,
\cite{CoMo:LIF}).

The homological formalism developed in our companion paper suggests a variant on the conventional approach via spectral triples. This variant is also suggested by the concrete picture developed in the early papers of Boll\'e et al, \cite{BGGSS:WKS}, and of Gesztesy and Simon, \cite{GeSi:TIW}. 
Loosely speaking, whereas the topological interpretation of the Fredholm index
is given by K-theory we view cyclic homology as the appropriate tool to replace K-theory in the case of non-Fredholm operators.
The main results presented here are about finding non-trivial examples of our homological index. 
These examples arise from the study of  Dirac type operators on the manifold $\mathbb R^{2n}$. 

Our motivation stems partly from the desire to understand higher dimensional examples of the operators considered in \cite{BGGSS:WKS}, and partly from
an ambition to probe the meaning of the homological index from the spectral point of view for Dirac type operators
on general non-compact manifolds, 
There is also motivation from magnetic Hamiltonians in dimensions greater than two and we will take this up elsewhere.

To explain our results we need some notation which we now present.

\subsection{Unbounded operators}\label{s:unb}
In this paper we will work with the following framework. First we double our Hilbert space setting $\C H^{(2)} := \C H \op \C H$, introducing a grading operator $\ga = \ma{cc}{1 & 0 \\ 0 & -1}$. We let $\dir_+$ be a closed densely defined operator on $\C H$ and form the odd selfadjoint operator $\dir := \ma{cc}{0 & \dir_- \\ \dir_+ & 0}$, where $\dir_- = (\dir_+)^*$. We will study a class of perturbations of $\dir$ of the form
\[
\C D := \ma{cc}{
0 & \dir_- + A^* \\
\dir_+ + A & 0
},
\]
where $A$ is a bounded operator on $\C H$. We want to study an invariant of $\C D$ by mapping to bounded operators using the Riesz map
\[
\C D \mapsto T = (\dir_+ + A)\big(1 + (\dir_-+A^*)(\dir_+ + A) \big)^{-1/2}.
\]
These bounded operators generate an algebra to which our homological theory \cite{CaKa:TIH} applies. To see how natural constraints arise on $T$ and its adjoint $T^*$ note that we have the identities
\[
\C D^2 = \ma{cc}{
(\dir_- + A^*)(\dir_+ + A) & 0 \\
0 & (\dir_+ + A)(\dir_- + A^*)
}
\]
and
\begin{equation}\label{eq:boures}
\begin{split}
1 - T T^* = ( 1 + \C D_+ \C D_- )^{-1}  \q \T{and} \q
1 - T^* T = ( 1 + \C D_- \C D_+ )^{-1},
\end{split}
\end{equation}
where $\C D_+ := \dir_+ + A$ and $\C D_- = (\C D_+)^* = \dir_- + A^*$.

We will be interested in the case where the following conditions hold:
\begin{assu}\label{a:unb}
\begin{enumerate}
\item The unbounded operator $\dir_+$ is normal, thus $\dir_+ \dir_- = \dir_- \dir_+$.
\item The bounded operators $A$ and $A^*$ have the domain of $\dir_+$ as an invariant subspace.
\item The sum of commutators $[\dir_+,A^*]  + [A,\dir_-] : \T{Dom}(\dir_-) \to \C H$ extends to a bounded operator on $\C H$.
\end{enumerate}
\end{assu}

Remark that the normality of $\dir_+$ entails that $\T{Dom}(\dir_+) = \T{Dom}(\dir_-)$ by \cite[Theorem 13.32]{Rud:FA}. 
%

Under the conditions in Assumption \ref{a:unb} we also obtain 
\[
\T{Dom}( \C D_+ \C D_- ) = \T{Dom}(\dir_+ \dir_-) = \T{Dom}(\dir_- \dir_+) = \T{Dom}( \C D_- \C D_+ )
\]
and furthermore that the identity
\[
\big[ (\dir_+ + A),(\dir_- + A^*)\big](\xi) = [\dir_+,A^*](\xi) - [\dir_-,A](\xi) + [A,A^*](\xi)
\]
holds for each vector $\xi \in \T{Dom}(\dir_+ \dir_-)$. This entails that $[\C D_+, \C D_-] : \T{Dom}(\dir_+\dir_-) \to \C H$ extends to a bounded operator on $\C H$.

Let $F \in \sL(\C H)$ denote the bounded extension of the commutator $[\C D_+, \C D_-]$. We may then impose compactness conditions on the difference
\[
(1 - T^* T) - (1 - T T^*) = ( 1 + \C D_+ \C D_- )^{-1} F ( 1 + \C D_- \C D_+ )^{-1}.
\]
In fact we will show in Section \ref{s:anodir} that there is a class of Dirac-type operators on $\rr^{2n}$, $n \in \nn$ which satisfy our hypotheses and lead to the condition 
\begin{equation}
\label{summ}
(1 - T^* T)^n - (1 - T T^*)^n \in \sL^1(\C H),
\end{equation}
where $\sL^1(\C H) \su \sL(\C H)$ denotes the ideal of trace class operators. The connection between the dimension of the underlying space $\rr^{2n}$ and the condition (\ref{summ})
is not evident in the earlier work \cite{Cal:AIO,BGGSS:WKS} but is natural from the point of view of spectral geometry.

The next definition is fundamental for the present text:

\begin{dfn}
Suppose that there exists an $n \in \nn$ such that $(1 - T^* T)^n - (1 - T T^*)^n$ is of trace class. By the \emph{homological index} of $T$ in degree $n$ we will understand the trace $\T{Tr}\big( (1 - T^* T)^n - (1 - T T^*)^n \big) \in \rr$. The homological index in degree $n$ is denoted by $\T{H-Ind}_n(T)$. When $T=\C D_+(1+\C D_-\C D_+)^{-1/2}$ we say that the homological index associated to $\C D_+$ is $\T{H-Ind}_n(T)$.
\end{dfn}

\subsection{Stability}
 One of the central problems to study concerns the stability or invariance properties of the homological index. More precisely, given a bounded operator $B : \C H \to \C H$, when can we say that the homological index (in degree $n$) associated to $\dir_+ + A + B$ exists and agrees with the homological index (in degree $n$) of $\dir_+ + A$? Since we are dealing with a genuinely non-compact situation, it would be naive to expect that such an invariance result holds for any bounded operator $B$, and this is one of the main differences between the homological index and the Fredholm index. Indeed it is known from the examples in 
\cite{BGGSS:WKS}, in connection with the Witten index, that in degree $n=1$ the homological index cannot be stable under general bounded perturbations $B$. In fact a stability result for the Witten index,  that gives the flavour of the complexity of the issue, is proved in \cite{GeSi:TIW}.

In the examples we are considering in this paper we must impose decay conditions at infinity on the bounded operator $B$, and this is naturally done using Schatten ideals. In our companion paper we gave a careful treatment of the invariance problem and, for the convenience of the reader, we state a simplified version of the main result here. We would like to emphasize though that the invariance result proved in the companion paper is more general (and for this reason we use distinct notation there) and allows us to deal with unbounded perturbations as well. 

\begin{thm}
Suppose that
\[
(1 + \C D_+ \C D_-)^{-j} B (1 + \C D_- \C D_+)^{-k-1/2} \, \, , \, \, (1 + \C D_- \C D_+)^{-j} B^* (1 + \C D_+ \C D_-)^{-k-1/2} \in \sL^{n/(j + k)}(\C H)
\]
for all $j,k \in \{0,\ldots,n\}$ with $1 \leq j + k \leq n$. Suppose furthermore that there exists an $\ep \in (0,1/2)$ such that
\[
B \cd (1 + \C D_- \C D_+)^{-n-1/2 + \ep} \, \, , \, \, B^* \cd (1 + \C D_+ \C D_-)^{-n-1/2 + \ep} \in \sL^1(\C H)
\]

Then the homological index of $T_B := (\C D_+ + B)(1 + (\C D_- + B^*)(\C D_+ + B))^{-1/2}$ exists in degree $n \in \nn$ if and only if the homological index of $T := \C D_+(1 + \C D_- \C D_+)^{-1/2}$ exists in degree $n \in \nn$. And in this case we have that
\[
\T{H-Ind}_n(T_B) = \T{H-Ind}_n(T)
\]
\end{thm}

For a proof of this invariance result we refer to \cite[Theorem 8.1]{CaKa:TIH}. We remark also that there is, in  \cite[Section 5]{CaKa:TIH}, a discussion of the homotopy invariance
of the homological index.

The reader may be puzzled by the complexity of the statement of this preceding theorem. This can be understood in 
part in terms of the topological meaning of the homological index. In the notation of the previous theorem we know from the examples in \cite{BGGSS:WKS} that the 
homological index cannot be invariant under compact perturbations of $T$. This is the first indication that, unlike the Fredholm index, the homological index is not directly associated to topological K-theory.  In our companion paper we show that the homological index is defined as a functional on certain homology groups of the algebra generated by $T$ and $T^*$. It is thus stable under perturbations that do not change the homology class and the hypotheses of the previous theorem stem from this fact.

On general manifolds we expect to see that the homological index depends not only on the topology,
but in fact is a spectral invariant that is much finer,  depending also on the geometry of the underlying space. This geometric dependence may be seen from our main result of this paper  (which we explain in the next subsection)
although we remark that the precise nature of the information that can be obtained from the homological index
remains to be determined.

In our companion paper, where the homological index was introduced, we did not address the question of non-triviality.
The main results of this paper can be seen to be a proof that the homological index is not trivial.
Most importantly, the homological viewpoint allows the systematic development of a higher dimensional theory into which the early work \cite{BGGSS:WKS, CaPi:IOG, GeSi:TIW} fits as the lowest degree case.

We now describe our main method.

\subsection{Scaling limits of the homological index}
Let $\dir_+ : \T{Dom}(\dir_+) \to \C H$ be a closed unbounded operator and let $A \in \sL(\C H)$ be a bounded operator. We will impose the conditions of Assumption \ref{a:unb} throughout this section. 

Let $\la \in (0,\infty)$ and consider the scaled operator $\la^{-1/2} \C D_+$. Let 
$$T_\la := \la^{-1/2}\C D_+(1 + \la^{-1} \C D_- \C D_+)^{-1/2}.$$ As in \eqref{eq:boures} we obtain 
\[
1 - T_\la T_\la^* = \la \cd (\la + \C D_+ \C D_-)^{-1} \q \T{and} \q 
1 - T_\la^* T_\la = \la \cd (\la + \C D_- \C D_+)^{-1}.
\]
We may then consider the difference of powers:
\[
(1- T_\la^* T_\la)^n - (1-T_\la T_\la^*)^n = \la^n \cd \big( (\la + \C D_- \C D_+)^{-n} - (\la + \C D_+ \C D_-)^{-n} \big)
\]

\emph{We are, in this paper, interested in the scaling limits of the homological index as the parameter $\la$ goes to either zero or infinity.} 
\vspace{3pt}

The scaling limit at zero is called the `Witten index' in the case $n=1$ in \cite{BGGSS:WKS} and \cite{GeSi:TIW}. What is known about this limit
can be found in these papers, the Carey-Pincus article and the recent preprint \cite{CGPST:WSS}. In particular we remark that when the operator $\C D$ is 
Fredholm and the scaling limit at zero exists then these articles demonstrate that it coincides with the Fredholm index.
The arguments in those papers can be adapted to establish an analogous result about the scaling limit
 at zero for $n>1$ in the Fredholm case. However in the non-Fredholm case little is known for $n>1$ and this is the subject of a future investigation.
Here our main focus will lie in the scaling limit at infinity. We will use the following terminology motivated by the examples in \cite{BGGSS:WKS, Cal:AIO}.

\begin{dfn}
Suppose that there exists an $n \in \nn$ such that $(1- T_\la^* T_\la)^n - (1-T_\la T_\la^*)^n \in \sL^1(\C H)$ for all $\la \in (0,\infty)$. The \emph{anomaly} in degree $n$ of the perturbed operator $\C D$ is then defined as the scaling limit 
\[
\T{Ano}_n(\C D) := \lim_{\la \to \infty} \T{H-Ind}_n(T_\la)
\]
whenever this limit exists.
\end{dfn}

If we were in a situation where the McKean-Singer formula  for the Fredholm index of $\C D_+$ was defined then the usual approach to obtaining a local formula for the index is to calculate the small time asymptotics of the trace of the heat kernel. This local formula is often referred to as the `anomaly'
in the physics literature.
It  corresponds in our situation to calculating the scaling limit at infinity. 

The use of the terminology `anomaly' can also be traced back to Callias \cite{Cal:AIO} where a resolvent  type formula, (that partly motivated the approach of  \cite{GeSi:TIW})  is introduced for the index of certain Fredholm  operators (of Dirac type) on $\mathbb R^{2n-1}$.
His index formula spurred many developments in the late 1970s and 80s. It differs from ours in two ways, first it applies to the odd dimensional case and second, the constraint (\ref{summ}) is absent: his formula depends only on the difference of resolvents. On the other hand the resolvent expansion methods of Callias  are analogous to our approach for calculating the anomaly.    

We remark that the result of Seeley which Callias states as Theorem 1 in \cite{Cal:AIO} may be used to demonstrate that the examples of Dirac type operators introduced in this paper are not Fredholm for general choices of connection.
Despite the fact that  we are dealing with the non-Fredholm situation, we are able to use the scaling limit at infinity 
(that is, the anomaly) as a means of
obtaining information on the non-triviality of the homological index.

{\it In the present paper we prove two main results.
First, we find conditions which ensure the existence of the homological index for a fixed degree $n \in \nn$. This will be carried out in Section \ref{ss:resana} and is stated as Theorem \ref{t:tradifpow}.

Second, and most importantly,  in Theorem \ref{t:locano} we give a {\it local formula} for the anomaly in terms of an integral of a $2n$-form on ${\mathbb R}^{2n}$ that is constructed as a function of the curvature of the connection coming from our Dirac type operator. The expression for the anomaly
may be seen to be non-zero by the calculation presented in Section \ref{nont}.}

The attentive reader will have noticed that our definitions of the anomaly, the Witten index (the scaling limit of the homological index at zero), as well as the homological index all depend on the degree $n \in \nn$. The precise way in which these spectral  invariants depend on the degree is yet to be determined. What can be inferred from the examples in this paper is that there exists a ``critical value'' of the degree such that the invariants are not well-defined in degrees strictly less than the critical value. For the Dirac-type operators, we consider in this paper, this critical value is exactly half the dimension of the underlying smooth manifold. Our computation of the anomaly for the examples in question shows moreover that the anomaly is independent of the degree above this critical value. We do not know whether this kind of behaviour holds in general for all of the spectral invariants introduced in this paper.
 
We remark that we impose very strong conditions on the connections we consider to reduce the number and complexity of the estimates needed to compute the anomaly.  It is, of course, important to establish the weakest conditions under which the anomaly exists, however, we have to leave this to another place.

\subsection{Motivation from physics examples}

We have indicated above the motivation that comes from the work of Boll\'e {\it et al},  Callias, Gesztesy-Simon and Witten.
The index problem posed in \cite{Cal:AIO} deals with non-self-adjoint operators on odd dimensional spaces. It fits into the framework of our companion paper \cite{CaKa:TIH} but we do not discuss this kind of example here. 

A second source of motivating examples is the study of 
Dirac operators coupled to connections in odd dimensions (in particular, the three dimensional case is of interest in condensed matter theory).
As we now explain these can be incorporated into our framework here.  In odd dimensions the fundamental invariant
is spectral flow. It is often claimed for models in condensed matter theory that spectral flow is a physically relevant invariant but
in many cases it is not obvious that one is dealing with Fredhom operators. On the other hand spectral flow is known to be related directly to Krein's spectral shift function and both can be expressed in terms of an index for a Dirac-type operator in even dimensions
as explained in \cite{GLMST}. We outline the argument. 

Starting from a pair of self-adjoint operators, denoted $A_{\pm}$  (in general unbounded), on a Hilbert space $\mathcal H$,  we introduce a
 `flow parameter' $s\in \mathbb R$ and a path of self adjoint operators  $A(s)$ with $\lim_{s\to\pm\infty} A(s) =A_\pm$ (where the limit is taken in an appropriate topology).
We then introduce a new operator acting on the `big Hilbert space'
$L^2(\mathbb R, \mathcal H\oplus \mathcal H)$ of the form
\[
\C D := \ma{cc}{
0 & \partial_s  + A(s) \\
-\partial_s +A(s)  & 0
} = \ma{cc}{0 & \C D_- \\ \C D_+ & 0}
\]
Under various assumptions one can relate the index of $\C D_+$ to the spectral flow for the path $A(s)$ when the latter path consists of Fredholm operators.
When the path $A(s), s\in \mathbb R$ consists of Dirac type operators on $\mathbb R^{2n-1}$ then $\C D$ is  of Dirac type on $\mathbb R^{2n}$. One may then ask (in the non-Fredholm case) the question of whether the
homological index for $T=\C D_+(1+\C D_-\C D_+)^{-1/2}$ is related to some generalised spectral flow, or spectral shift function, for the pair $A_{\pm}$.
 This kind of problem arises when $A_\pm$ are magnetic Dirac type operators. So for example if $A_\pm$ are   operators in three dimensions then   the degree $n=2$ homological index comes into play for $\C D_+$.
  We are currently investigating this application and will report on the outcome elsewhere.

\section{Preliminaries}

\subsection{Notation}
Let $\dir_+ : \T{Dom}(\dir_+) \to \C H$ be a closed unbounded operator and let $A \in \sL(\C H)$ be a bounded operator. The conditions in Assumption \ref{a:unb} will be in effect.

\begin{notation}\label{n:Iresana}
We will use the following notation for various unbounded operators related to $A$ and $\dir_+$:
\[
\C D_+ := \dir_+ + A : \T{Dom}(\dir_+) \to \C H, \quad\quad
\C D_- := (\C D_+)^* = \dir_- + A^* : \T{Dom}(\dir_-) \to \C H \]
\[
\De := \dir_+ \dir_- = \dir_- \dir_+ : \T{Dom}(\De) \to \C H, \]
\[
\De_1 := \C D_+ \C D_- : \T{Dom}(\De) \to \C H, \quad
\De_2 := \C D_- \C D_+ : \T{Dom}(\De) \to \C H
\]
The bounded extension of the commutator $[\C D_+,\C D_-] : \T{Dom}(\De) \to \C H$ is denoted by $F \in \sL(\C H)$.
\end{notation}

Remark that we have the expressions
\[
\begin{split}
\De_1 - \De & = A \dir_- + \dir_+ A^* + A A^* : \T{Dom}(\De) \to \C H, \\
\De_2 - \De & = A^* \dir_+ + \dir_- A + A^* A : \T{Dom}(\De) \to \C H.
\end{split}
\]

\begin{notation}\label{n:IIresana}
Introduce the following notation
\[
\begin{split}
V_1 := A \dir_- +  \dir_+ A^* + A A^* : \T{Dom}(\dir_+) \to \C H, \\
V_2 := A^* \dir_+ + \dir_- A + A^* A : \T{Dom}(\dir_+) \to \C H
\end{split}
\]
for the extensions to $\T{Dom}(\dir_+)$ of $\De_1 - \De$ and $\De_2 - \De$, respectively. Furthermore, we let
\[
Y_1(\la) := V_1 \cd (\la + \De)^{-1} : \C H \to \C H \q Y_2(\la) := V_2 \cd (\la+\De)^{-1} : \C H \to \C H
\]
for all $\la \in (0,\infty)$.
\end{notation}

It is important to notice that the linear operators $Y_1(\la) : \C H \to \C H$ and $Y_2(\la) : \C H \to \C H$ are unbounded operators in general. This is due to the fact that the commutator $[\dir_-,A]$ does \emph{not} always extend to a bounded operator. This condition is not even satisfied for the perturbations of the Dirac operator on $\rr^{2n}$ which we consider in Section \ref{s:anodir}. We will however impose differentiability conditions on $A$ later on which imply that $Y_1(\la)$ and $Y_2(\la)$ are bounded. These differentiability conditions will be the subject of the next subsection.

\subsection{Quantum differentiability}\label{ss:quadif}
While this subsection concerns classical pseudo-differential operators we anticipate noncommutative applications and hence frame the definitions accordingly.
In the first part of this  subsection we will consider a selfadjoint positive unbounded operator $\De : \T{Dom}(\De) \to \C H$ on the Hilbert space $\C H$. The functional calculus for selfadjoint unbounded operators provides us with a scale of dense subspaces $\C H^s := \T{Dom}(\De^{s/2}) \su \C H$, $s \in [0,\infty)$. Each subspace $\C H^s$ becomes a Hilbert space in its own right when equipped with the inner product
\[
\inn{\cd,\cd}_s : (\xi,\eta) \mapsto \inn{\xi,\eta} + \inn{\De^{s/2}\xi,\De^{s/2}\eta} \q \xi,\eta \in \C H^s.
\]
Notice also that $\C H^s \su \C H^r$ whenever $s \geq r$. This scale of Hilbert spaces plays the role of Sobolev spaces in noncommutative geometry.

Following Connes and Moscovici, \cite[Appendix B]{CoMo:LIF}, we consider the $*$-subalgebra $\T{OP}^0 \su \sL(\C H)$ consisting of bounded operators $T$ for which
\begin{enumerate}
\item The domain of $\De^{n/2}$ is an invariant subspace for all $n \in \nn$.
\item The iterated commutator $\de^n(T) : \T{Dom}(\De^{n/2}) \to \C H$ extends to a bounded operator on $\C H$ for all $n \in \nn$.
\end{enumerate}
Here $\de^n(T) : \T{Dom}(\De^{n/2}) \to \C H$ is defined recursively by $\de(T) = [\De^{1/2},T]$ and $\de^n(T) = [\De^{1/2},\de^{n-1}(T)]$.

For each $\la \in (0,\infty)$ and each $m \in \nn_0$ we define the algebra automorphism 
\[
\si_\la^m : \T{OP}^0 \to \T{OP}^0 \q \si_\la^m(T) := (\la + \De)^m T (\la + \De)^{-m}.
\]
Notice that the inverse $\si_\la^{-m} : \T{OP}^0 \to \T{OP}^0$ of $\si_\la^m$ can be defined by $\si_\la^{-m}(T) := \big( \si_\la^m(T^*) \big)^*$. The element $\si_\la^{-m}(T) \in \T{OP}^0$ then agrees with the bounded extension of $(\la + \De)^{-m} T (\la + \De)^m : \T{Dom}(\De^m) \to \C H$.

We record the following:

\begin{lemma}\label{l:difact}
Let $m \in \zz$. For each $T \in \T{OP}^0$ we have the estimate
\[
\| \si^m_\la(T) - T\| = O(\la^{-1/2}) \q \T{as }\, \la \to \infty
\]
\end{lemma}
\begin{proof}
It is enough to prove the assertion for $m \in \nn_0$. The easiest proof then runs by induction. The statement is trivial for $m= 0$. Thus suppose that it holds for some $m_0 \in \nn_0$.

Let $T \in \T{OP}^0$. Then
\[
\begin{split}
\si^{m_0 + 1}_\la(T) & = (\la + \De)^{m_0 + 1} T (\la + \De)^{-m_0 - 1} \\
& = \si^{m_0}_\la(T) + (\la + \De)^{m_0} [\De,T](\la + \De)^{-m_0 - 1}.
\end{split}
\]
The fact that $T \in \T{OP}^0$ implies that $[\De,T](1+ \De)^{-1/2}$ extends to an element $R(T) \in \T{OP}^0$. We thus have that
\[
(\la + \De)^{m_0} [\De,T] (\la + \De)^{-m_0 - 1} = \si_\la^{m_0}\big(  R(T) \big) \cd (1+\De)^{1/2} \cd (\la + \De)^{-1}.
\]

Since $\|(1+\De)^{1/2} \cd (\la + \De)^{-1}\| = O(\la^{-1/2})$ we obtain the estimate $\|\si^{m_0 + 1}_\la(T) - T\| = O(\la^{-1/2})$ by applying the induction hypothesis to $T$ and $R(T)$.
\end{proof}

We will now return to the setup introduced in the beginning of subsection \ref{s:unb}. The conditions on the closed operator $\dir_+$ and the bounded operator $A$ which are stated in Assumption \ref{a:unb} will in particular be in effect. We will then consider the unbounded operators $Y_1(\la) = V_1 (\la + \De)^{-1}$ and $Y_2(\la) = V_2(\la + \De)^{-1}$ introduced in Notation \ref{n:IIresana}.

\begin{lemma}\label{l:disanaest}
Suppose that $A \in \T{OP}^0$. Then $Y_1(\la), Y_2(\la) \in \T{OP}^0$ for all $\la \in (0,\infty)$. Furthermore, we have the estimate
\[
\|Y_1(\la)\| = \|Y_2(\la)\| = O(\la^{-1/2}) \q \T{as }\, \la \to \infty
\]
\end{lemma}
\begin{proof}
Let $\la \in (0,\infty)$. To ease the notation, let $F_+(\la) := \dir_+ (\la + \De)^{-1}$ and $F_-(\la) := \dir_-(\la+\De)^{-1}$. Compute as follows,
\[
\begin{split}
V_1 \cd (\la + \De)^{-1} & = A \dir_- \cd (\la+\De)^{-1} + \dir_+ A^* \cd (\la + \De)^{-1} + AA^* \cd (\la+\De)^{-1} \\
& = A \cd F_-(\la) + AA^* \cd (\la+\De)^{-1} + F_+(\la) \cd \si_\la(A^*).
\end{split}
\]
Since each of the bounded operators $A \cd F_-(\la)$, $AA^* \cd (\la+\De)^{-1}$ and $F_+(\la) \cd \si_\la(A^*)$ lies in $\T{OP}^0$ this shows that $Y_1(\la) \in \T{OP}^0$ as well.

The desired estimate on $\|Y_1(\la)\|$ follows from the above identities and Lemma \ref{l:disanaest} since $\|F_+(\la)\| = \|F_-(\la)\| = O(\la^{-1/2})$ as $\la \to \infty$.

A similar argument proves the claims on $Y_2(\la)$ as well.
\end{proof}

Recall that a unital $*$-subalgebra $\sA$ of a unital $C^*$-algebra $A$ is said to be \emph{closed under
the  holomorphic functional calculus} when for each $x \in \sA$ and each holomorphic function $f$ on the spectrum of $x$ in $A$ we have that $f(x) \in \sA$.

The following result is well known but we prove it here for lack of an adequate reference:

\begin{lemma}\label{l:resqua}
The unital $*$-subalgebra $\T{OP}^0 \su \sL(\C H)$ is closed under the holomorphic functional calculus.
\end{lemma}
\begin{proof}
For each $N \in \nn$, let $\T{OP}^0_N$ denote the unital $*$-algebra consisting of the bounded operators $T$ such that
\begin{enumerate}
\item The domain of $\De^{n/2}$ is an invariant subspace for all $n \in \{1,\ldots,N\}$.
\item The iterated commutator $\de^n(T) : \T{Dom}(\De^{n/2}) \to \C H$ extends to a bounded operator on $\C H$ for all $n \in \{1,\ldots,N\}$.
\end{enumerate}
The unital $*$-algebra $\T{OP}^0_N$ becomes a unital Banach $*$-algebra when equipped with the norm
\[
\|\cd\|_N : T \mapsto \sum_{n=0}^N \frac{1}{n!} \|\de^n(T)\|.
\]
It then follows from \cite[Proposition 3.12]{BlCu:DNS} that $\T{OP}^0_N$ is closed under holomorphic functional calculus. But this proves the lemma since $\T{OP}^0 = \bigcap_{N=1}^\infty \T{OP}^0_N$.
\end{proof}

The next statement is now a consequence of Lemma \ref{l:disanaest} and Lemma \ref{l:resqua}.

\begin{prop}\label{p:anainv}
Suppose that $A \in \T{OP}^0$. Then there exists a constant $C > 0$ such that $\big(1 + Y_1(\la) \big)^{-1}$ and $\big( 1 + Y_2(\la)\big)^{-1}$ are well-defined elements in $\T{OP}^0$ for all $\la \geq C$.
\end{prop}

\section{Resolvent expansions}\label{ss:resana}
The object of this section is to find verifiable criteria which imply that the homological index exists. We will thus be studying expressions of the form
\[
(z + \De_2)^{-n} - (z+ \De_1)^{-n}
\]
for a natural number $n \in \nn$ and a complex number $z \in \cc\sem (-\infty,0]$. The notation used throughout will be the one introduced in Notation \ref{n:Iresana} and Notation \ref{n:IIresana}. Furthermore, the conditions in Assumption \ref{a:unb} will be in effect.

The main idea is to approximate the operators $\De_1 = \C D_+ \C D_-$ and $\De_2 = \C D_- \C D_+$ by the Laplacian $\De = \dir_+ \dir_-$ using perturbation theoretic techniques.

 We start by stating some preliminary lemmas. The proofs are well-known and will not be repeated here.

\begin{lemma}\label{l:difpowexp}
Let $n \in \nn$ and let $z \in \cc\sem (-\infty,0]$. Then
\[
(z + \De_2)^{-n} - (z + \De_1)^{-n} = \sum_{i=0}^{n-1} (z + \De_1)^{-i-1} \cd F \cd (z + \De_2)^{-(n-i)}.
\]
\end{lemma}

\begin{lemma}\label{l:resexp}
Let $H_1,H_2$ be selfadjoint positive unbounded operators on $\C H$ with $\T{Dom}(H_1) = \T{Dom}(H_2)$. Let $z,z_0 \in \cc\sem (-\infty,0]$ and suppose that the linear operator $(H_1 - H_2 + z - z_0)(z_0 + H_2)^{-1} : \C H \to \C H$ is bounded with $\|(H_1 - H_2 + z- z_0)(z_0 + H_2)^{-1}\| < 1$. We then have the identity
\[
(z + H_1)^{-1} = (z_0 + H_2)^{-1} \cd \big( 1 + (H_1 - H_2 + z-z_0)(z_0 + H_2)^{-1} \big)^{-1}.
\]
\end{lemma}

The next result implies that the trace of the quantity $$z^n \big( (z+ \C D_- \C D_+)^{-n} - (z + \C D_+ \C D_-)^{-n} \big),$$
which is computing the homological index, depends holomorphically on the complex parameter $z \in \cc\sem (-\infty,0]$.

\begin{prop}\label{p:holhomind}
Let $n \in \nn$ and let $i \in \{0,\ldots,n-1\}$. Suppose that there exists a $z_0 \in \cc\sem (-\infty,0]$ such that
\[
(z_0 + \De_1)^{-i-1} \cd F \cd (z_0 + \De_2)^{-(n-i)} \in \sL^1(\C H).
\]
Then the operator $(z + \De_1)^{-i-1} \cd F \cd (z + \De_2)^{-(n-i)}$ is of trace class for all $z \in \cc\sem (-\infty,0]$. Furthermore, the map
\[
\cc\sem (-\infty,0] \to \sL^1(\C H) \q
z \mapsto (z + \De_1)^{-i-1} \cd F \cd (z + \De_2)^{-(n-i)}
\]
is holomorphic with respect to the trace norm on $\sL^1(\C H)$.
\end{prop}
\begin{proof}
Let $z \in \cc\sem (-\infty,0]$. Remark then that
\begin{equation}\label{eq:powexp}
(z + \De_1)^{-i-1} = (z_0 + \De_1)^{i+1} \cd (z + \De_1)^{-i-1}  \cd (z_0 + \De_1)^{-i-1}.
\end{equation}
This shows that the map $z \mapsto (z + \De_1)^{-i-1}$ factorizes as a product $z \mapsto H(z) \cd (z_0 + \De_1)^{-i-1}$, where $z \mapsto H(z)$, $\cc\sem (-\infty,0] \to \sL(\C H)$ is holomorphic.

A similar computation shows that the map $z \mapsto (z + \De_2)^{-(n-i)}$ factorizes as a product $z \mapsto (z_0 + \De_2)^{-(n-i)} \cd K(z)$, where $z \mapsto K(z)$, $\cc\sem (-\infty,0] \to \sL(\C H)$, is holomorphic.

As a consequence, we obtain 
\[
(z + \De_1)^{-i-1} \cd F \cd (z + \De_2)^{-(n-i)}
= H(z) \cd (z_0 + \De_1)^{-i-1} \cd F \cd (z_0 + \De_2)^{-(n-i)} \cd K(z)
\]
for all $z \in \cc\sem (-\infty,0]$. This proves the proposition.
\end{proof}

In the following proposition, we introduce an expansion of the quantity $$(\la + \De_2)^{-n} - (\la + \De_1)^{-n}.$$ This expansion will serve us in two ways. First, it allows us to identify a good criterion for the existence of the homological index and second,
 it will help us to give a concrete computation of this quantity (or more precisely of its scaling limit at infinity).

To ease the notation, let $X_1(\la) := (1 + Y_1(\la))^{-1}$ and $X_2(\la) := (1 + Y_2(\la))^{-1}$ whenever these quantities make sense as bounded operators.

\begin{prop}\label{p:resexprea}
Suppose that $A \in \T{OP}^0$ and let $n \in \nn$. Then there exists a constant $C > 0$ such that
\[
\begin{split}
& (\la + \De_1)^{-i-1} \cd F \cd (\la + \De_2)^{-(n-i)} \\
& \q = \si^{-1}_\la\big(X_1(\la)\big) \clc \si^{-i-1}_\la \big( X_1(\la) \big) \cd (\la + \De)^{-i-1} \cd F \\
& \qqq \cd (\la + \De)^{-(n-i)} \cd \si^{n-i-1}_\la \big(X_2(\la) \big) \clc X_2(\la)
\end{split}
\]
for all $\la \geq C$ and all $i \in \{0,\ldots,n-1\}$
\end{prop}
\begin{proof}
Choose the constant $C > 0$ as in Proposition \ref{p:anainv} and let $\la \geq C$. It follows that $X_1(\la) = \big(1+Y_1(\la) \big)^{-1}$ and $X_2(\la) = \big(1 + Y_2(\la)\big)^{-1}$ lie in $\T{OP}^0$.

Let $i \in \{0,\ldots,n-1\}$. The result of Lemma \ref{l:resexp} yields that
\begin{equation}
\begin{split}\label{eq:firres}
& (\la + \De_1)^{-i-1} \cd F \cd (\la + \De_2)^{-(n-i)} \\
& \q = \big( (\la + \De)^{-1} X_1(\la) \big)^{i+1} \cd F \cd \big(
(\la + \De)^{-1} \cd X_2(\la) \big)^{n-i}.
\end{split}
\end{equation}

But this implies the statement of the proposition since
\[
\begin{split}
& \big( (\la + \De)^{-1} \cd X_1(\la) \big)^{i+1} = \si^{-1}_\la\big( X_1(\la) \big) \clc \si^{-i-1}_\la \big( X_1(\la) \big) \cd (\la + \De)^{-i-1} 
\q \T{and} \\
& \big( (\la + \De)^{-1} \cd X_2(\la) \big)^{n-i} = (\la + \De)^{-n + i} \si^{n-i-1}_\la \big( X_2(\la) \big) \cd \ldots \cd X_2(\la),
\end{split}
\]
by the definition of the automorphisms $\si^m_\la : \T{OP}^0 \to \T{OP}^0$, $m \in \zz$.
\end{proof}

We are now ready to state and prove the announced criterion for the existence of the homological index.

\begin{prop}\label{p:tradifpow}
Suppose that $A \in \T{OP}^0$ and that there exists an $n_0 \in \nn$ such that
\[
(1+ \De)^{-i-1} \cd F \cd (1+\De)^{-n_0+i} \in \sL^1(\C H) \q \T{for all }\, i \in \{0,1,\ldots,n_0-1\}.
\]
Let $n \geq n_0$ and let $i \in \{0,1,\ldots,n-1\}$. Then
\[
(z+\De_1)^{-i-1} F (z+\De_2)^{-n+i} \in \sL^1(\C H) \q 
\T{for all }\, z \in \cc\sem (-\infty,0].
\]
Furthermore, the map
\[
\cc\sem (-\infty,0] \to \sL^1(\C H) \q z \mapsto (z+\De_1)^{-i-1} F (z+\De_2)^{-n+i}
\]
is holomorphic in trace norm.
\end{prop}
\begin{proof}
Choose a constant $C > 0$ as in Proposition \ref{p:resexprea} and let $\la \geq C$. We obtain 
\[
\begin{split}
& (\la + \De_1)^{-i-1} \cd F \cd (\la + \De_2)^{-(n-i)} \\
& \q = \si^{-1}_\la\big(X_1(\la)\big) \clc \si^{-i-1}_\la \big( X_1(\la) \big) \cd (\la + \De)^{-i-1} \cd F \\
& \qqq \cd (\la + \De)^{-(n-i)} \cd \si^{n-i-1}_\la \big(X_2(\la) \big) \clc X_2(\la)
\end{split}
\]
This implies that $(\la + \De_1)^{-i-1} \cd F \cd (\la + \De_2)^{-(n-i)} \in \sL^1(\C H)$. Indeed, when $i \in \{0,\ldots,n_0 - 1\}$ we have that
\[
(\la + \De)^{-i-1} \cd F \cd (\la + \De)^{-(n-i)} = (\la + \De)^{-i-1} \cd F \cd (\la + \De)^{-n_0 + i} \cd (\la +\De)^{n_0 - n} \in \sL^1(\C H)
\]
whereas when $i \in \{n_0,\ldots,n-1\}$ we have that
\[
\begin{split}
& (\la + \De)^{-i-1} \cd F \cd (\la + \De)^{-(n-i)}\\
& \q  = (\la + \De)^{n_0 - i - 1} \cd (\la + \De)^{-n_0} \cd F \cd (\la + \De)^{-1} \cd (\la + \De)^{-n+1 + i} 
\in \sL^1(\C H).
\end{split}
\]
The statement of the proposition is now a consequence of Proposition \ref{p:holhomind}.
\end{proof}

Let us summarize what we have proved thus far:

\begin{thm}\label{t:tradifpow}
Let $\dir_+ : \T{Dom}(\dir_+) \to \C H$ be a closed operator and let $A$ be a bounded operator that 
together satisfy
the conditions in Assumption \ref{a:unb}. Suppose that $A \in \T{OP}^0$  (with respect to the Laplacian $\De$) and that there exists an $n_0 \in \nn$ such that $(1+\De)^{-i-1} F (1+ \De)^{-n_0+i}$ is of trace class for all $i \in \{0,\ldots,n_0-1\}$. 

Then the homological index $\T{H-Ind}_n(T_\la)$ exists for all $\la \in (0,\infty)$ and all $n \geq n_0$. Furthermore, for each $n \geq n_0$, the map $z \mapsto z^n \cd \T{Tr}\big( (z + \De_2)^{-n} - (z+ \De_1)^{-n} \big)$ provides a holomorphic extension of $\la \mapsto \T{H-Ind}_n(T_\la)$ to the open set $\cc\sem(-\infty,0]$.
\end{thm}

\section{Large scale expansions of the homological index}\label{s:larscaexp}
As previously $\C H$ is a Hilbert space. Define the projections 
\[
P_+ := \ma{cc}{
1 & 0 \\
0 & 0
} \, \T{ and }\, P_- := \ma{cc}{
0 & 0 \\
0 & 1
} : \C H \op \C H \to \C H \op \C H.
\]

Consider the vector subspace $\sL^1_s(\C H) \su \sL(\C H \op \C H)$ consisting of square matrices $T = \ma{cc}{T_{11} & T_{12} \\ T_{21} & T_{22}}$ of bounded operators with $T_{11} - T_{22} \in \sL^1(\C H)$. This vector space becomes a Banach space when equipped with the norm
\[
\|\cd\|_1^s : \sL^1_s(\C H) \to [0,\infty) \q \|\cd\|_1^s : T \mapsto \|T\| + \|P_+ T P_+ - P_- T P_-\|_1.
\]
The super trace
\[
\T{Tr}_s : \sL^1_s(\C H) \to \cc \q \T{Tr}_s : T \mapsto \T{Tr}\big( P_+ T P_+ - P_- T P_-\big)
\]
defines a bounded linear functional on this Banach space.
\vspace{3pt}

Let now $A : \C H \to \C H$ be a bounded operator and let $\dir_+ : \T{Dom}(\dir_+) \to \C H$ be a closed unbounded operator. For the rest of this section the conditions in Assumption \ref{a:unb} will be in effect. Furthermore, we will impose the following:

\begin{assu}\label{a:tra}
Assume that $A \in \T{OP}^0$ and that there exists an $n_0 \in \nn$ such that $(1 + \De)^{-i-1} \cd F \cd (1+\De)^{-n_0 +i+\ep}$ is of trace class for all $i \in \{-1,0,1,\ldots,n_0-1\}$ and all $\ep \in (0,1)$.
\end{assu}

Let us fix a natural number $n \in \{n_0,n_0 + 1,\ldots\}$.

As a consequence of Theorem \ref{t:tradifpow} we have a well-defined homological index
\[
\la^n \T{Tr}\big( (\la + \De_2)^{-n} - (\la + \De_1)^{-n} \big) = \T{H-Ind}_n(T_\la)
\]
for all $\la \in (0,\infty)$. This quantity can be reinterpreted as follows.  Consider the $n^{\T{th}}$ power of the resolvent of the scaled operator $\la^{-1}\C D^2$,
\begin{equation}\label{eq:ressqu}
\begin{split}
(1 + \la^{-1}\C D^2)^{-n} 
& = \ma{cc}{
(1 + \la^{-1}\C D_- \C D_+)^{-n} & 0 \\
0 & (1 + \la^{-1}\C D_+\C D_-)^{-n}
} \\
& = \la^n \ma{cc}{
(\la + \De_2)^{-n} & 0 \\
0 & (\la + \De_1)^{-n}
}.
\end{split}
\end{equation}
The homological index in degree $n$ is then obtained as the super trace of the operator in \eqref{eq:ressqu}:
\[
\begin{split}
\T{H-Ind}_n(T_\la) & = \T{Tr}\big( P_+(1 + \la^{-1} \C D^2)^{-n} P_+ - P_-(1 + \la^{-1} \C D^2)^{-n}P_- \big) \\
& := \T{Tr}_s \big( (1 + \la^{-1}\C D^2 )^{-n} \big)
\end{split}
\]
Remark that this super trace makes sense even though we do not assume that each term on the diagonal is of trace class.

The first aim of this section is to obtain an expansion of the resolvent $(1 + \la^{-1} \C D^2)^{-n} \in \sL^1_s(\C H)$ when the scaling parameter $\la > 0$ is large. The standard terms in this expansion have the following form:

\begin{notation}\label{n:omedeg}
For each $J = (j_1,\ldots,j_n) \in \nn_0^n$ and each $\la > 0$, define the bounded operators
\[
\begin{split}
\om_1(J,\la) & := (-1)^{j_1 \plp j_n}(\la + \De)^{-1} Y_1(\la)^{j_1} \clc (\la + \De)^{-1} Y_1(\la)^{j_n} \q \T{and} \\
\om_2(J,\la) & := (-1)^{j_1 \plp j_n}(\la + \De)^{-1} Y_2(\la)^{j_1} \clc (\la + \De)^{-1} Y_2(\la)^{j_n}
\end{split}
\]
on the Hilbert space $\C H$.
\end{notation}

The main technical result for obtaining the desired expansion of the resolvent is a trace norm estimate on the difference $\om_2(J,\la) - \om_1(J,\la)$, for $J \in \nn_0^n$ and large scaling parameters $\la$. This is contained in the next lemma.

\begin{lemma}\label{l:omeest}
The diagonal operator
\[
\om(J,\la) := \ma{cc}{
\om_2(J,\la) & 0 \\
0 & \om_1(J,\la)
}
\]
lies in the Banach space $\sL^1_s(\C H)$ for all $J \in \nn_0^n$ and all $\la >0$. Furthermore, for each $\ep \in (0,1)$, there exist constants $C,K>0$ such that
\[
\|\om(J,\la)\|_1^s \leq (j_1 \plp j_n + 1) \cd K^{j_1 \plp j_n} \cd \la^{n_0 - n -(j_1 \plp j_n)/2 + 1/2 - \ep} 
\]
for all $J = (j_1,\ldots,j_n) \in \nn_0^n$ and all $\la \geq C$.
\end{lemma}
\begin{proof}
Let $J \in \nn_0^n$ and let $\la > 0$. Let $k \in \{1,\ldots,n\}$ and let $l \in \{1,\ldots,j_k\}$. To show that $\om_2(J,\la) - \om_1(J,\la) \in \sL^1(\C H)$ it is enough to verify that
\[
\begin{split}
& (\la + \De)^{-1} Y_2(\la)^{j_1} \clc (\la + \De)^{-1} Y_2(\la)^{j_{k - 1}} \\
& \q \cd (\la + \De)^{-1} Y_2(\la)^{l-1} \cd \big( Y_2(\la) - Y_1(\la) \big) Y_1(\la)^{j_k -l} \\
& \qq \cd (\la + \De)^{-1} Y_1(\la)^{j_{k + 1}} \clc (\la + \De)^{-1} Y_1(\la)^{j_n} \in \sL^1(\C H).
\end{split}
\]
Using Lemma \ref{l:disanaest} this bounded operator can be rewritten as
\[
\begin{split}
& -\si_\la^{-1}\big(Y_2(\la)^{j_1} \big) \clc \si_\la^{-k+1}\big( Y_2(\la)^{j_{k-1}} \big)
\cd \si_\la^{-k}\big( Y_2(\la)^{l-1} \big) \\
& \q \cd (\la + \De)^{-k} \cd F \cd (\la + \De)^{-n+k-1} \\
& \qq \cd \si_\la^{n-k}\big( Y_1(\la)^{j_k - l} \big) \cd \si_\la^{n-k-1}\big( Y_1(\la)^{j_{k+1}} \big) \clc Y_1(\la)^{j_n},
\end{split}
\]
where we recall that $F \cd (\la + \De)^{-1} = Y_1(\la) - Y_2(\la)$. This proves that $\om(J,\la) \in \sL^1_s(\C H)$ since the conditions of Assumption \ref{a:tra} are in effect. See also the proof of Proposition \ref{p:tradifpow}.

To prove the norm estimate, we first note that there exist constants $C_1,K_1 > 0$ such that
\begin{equation}\label{eq:openorest}
\|\om_1(J,\la) \| \, , \, \|\om_2(J,\la)\| \leq \la^{-n} K_1^{j_1 \plp j_n} \cd \la^{-(j_1\plp j_n)/2}
\end{equation}
for all $\la \geq C_1$ and all $J = (j_1,\ldots,j_n) \in \nn_0^n$. See Lemma \ref{l:disanaest}.

Let now $\ep \in (0,1)$. By our standing Assumption \ref{a:tra} we may choose a constant $K_3 > 0$ such that
\begin{equation}\label{eq:tranorest0}
\| (\la + \De)^{-k} \cd F \cd (\la + \De)^{-n+k-1} \|_1 \leq K_3 \cd \la^{n_0 - n - \ep}
\end{equation}
for all $k \in \{1,\ldots,n\}$ and all $\la \geq 1$.

Furthermore, by an application of Lemma \ref{l:disanaest} and Lemma \ref{l:difact} we may choose constants $C_2,K_2 > 0$ such that
\begin{equation}\label{eq:tranorest}
\begin{split}
& \big\| \si_\la^{-1}\big(Y_2(\la)^{j_1} \big) \clc \si_\la^{-k+1}\big( Y_2(\la)^{j_{k-1}} \big)
\cd \si_\la^{-k}\big( Y_2(\la)^{l-1} \big) \\
& \qq \cd (\la + \De)^{-k} \cd F \cd (\la + \De)^{-n+k-1} \\
& \qqq \cd \si_\la^{n-k}\big( Y_1(\la)^{j_k - l} \big) \cd \si_\la^{n-k-1}\big( Y_1(\la)^{j_{k+1}} \big) \clc Y_1(\la)^{j_n} \big\|_1 \\
& \q \leq K_2^{j_1 \plp j_n -1} \cd \la^{-(j_1 \plp j_n - 1)/2} \cd \| (\la + \De)^{-k} \cd F \cd (\la + \De)^{-n+k-1} \|_1
\end{split}
\end{equation}
for all $J \in \nn_0^n$, $k \in \{1,\ldots,n\}$, $l \in \{1,\ldots,j_k\}$ and all $\la \geq C_2$.

A combination of \eqref{eq:openorest}, \eqref{eq:tranorest0} and \eqref{eq:tranorest} proves the norm estimate on $\om(J,\la)$.
\end{proof}

Let $\C B : \nn_0^n \to \nn_0$ denote the addition, $\C B : (j_1,\ldots,j_n) \mapsto j_1 \plp j_n$. The level sets will be written as $\C B_k := \C B^{-1}(\{k\})$, for all $k \in \nn_0$.

We are now ready to rewrite the homological index using an infinite series of operators of the form $\om(J,\la)$, where $J \in \nn_0^n$ and the scaling parameter $\la$ is large. This expansion will be the main tool for computing the anomaly of Dirac type operators in Section \ref{s:locano}.

\begin{prop}\label{p:degexp}
Let $\dir_+ : \T{Dom}(\dir_+) \to \C H$ be a closed operator and let $A : \C H \to \C H$ be a bounded operator which together satisfy the conditions in Assumption \ref{a:unb} and Assumption \ref{a:tra} for some $n_0 \in \nn$. Let $n \in \{n_0,n_0 +1,\ldots\}$. 

Then there exists a constant $C > 0$ such that
\[
(1 + \la^{-1}\C D^2)^{-n} = \la^n \cd \sum_{k = 0}^\infty \sum_{J \in \C B_k} \om(J,\la)
\]
for all $\la \geq C$, where the sum converges absolutely in $\sL^1_s(\C H)$. In particular, we have that
\[
\T{H-Ind}_n(T_\la) = \la^n \cd \sum_{k = 0}^\infty \sum_{J \in \C B_k} \T{Tr}_s\big( \om(J,\la) \big)
\]
for all $\la \geq C$, where the sum converges absolutely.
\end{prop}
\begin{proof}
Choose constants $C_1,K_1 > 0$ as in Lemma \ref{l:omeest} (with $\ep = 3/4$). We then have that
\[
\begin{split}
\la^n \cd \sum_{k=0}^\infty \sum_{J \in \C B_k} \|\om(J,\la)\|_1^s 
& \leq \sum_{k=0}^\infty \sum_{J \in \C B_k} (k+1) \cd K_1^k \cd \la^{n_0 -k/2 - 1/4} \\
& = \la^{n_0 - 1/4} \sum_{k=0}^\infty (k+1) \cd (K_1/\sqrt{\la})^k \cd {n + k-1 \choose n-1},
\end{split}
\]
for all $\la \geq C_1$. This shows that the sum $\la^n \cd \sum_{k = 0}^\infty \sum_{J \in \C B_k} \om(J,\la)$ converges absolutely in $\sL^1_s(\C H)$ whenever $\la > \T{max}\{C_1,K^2_1\}$.

Now, by Lemma \ref{l:disanaest} we may choose a constant $C_2 > 0$ such that $\|Y_1(\la)\| \, , \, \|Y_2(\la)\| < 1$, for all $\la \geq C_2$. An application of Lemma \ref{l:resexp} then implies that
\[
\begin{split}
(\la + \De_l)^{-n} & = \Big( (\la + \De)^{-1} \big( 1 + Y_l(\la) \big)^{-1} \Big)^n 
= \Big( (\la + \De)^{-1} \cd \sum_{j=0}^\infty (-1)^j Y_l(\la) \Big)^n \\
& = \sum_{J \in \nn_0^n} \om_l(J,\la) = \sum_{k=0}^\infty \sum_{J \in \C B_k} \om_l(J,\la),
\end{split}
\]
for all $\la \geq C_2$ and $l \in \{1,2\}$, where the sums converge absolutely in operator norm. A combination of these observations proves the first part of the proposition since
\[
(1 + \la^{-1} \C D^2)^{-n} = \la^n \cd \ma{cc}{ (\la + \De_2)^{-n} & 0 \\
0 & (\la + \De_1)^{-n}
},
\]
for all $\la > 0$.

The second statement of the proposition is now a consequence of the boundedness of the super trace $\T{Tr}_s : \sL^1_s(\C H) \to \cc$ since $\T{H-Ind}_n(T_\la) := \T{Tr}_s\big( (1 + \la^{-1} \C D^2)^{-n} \big)$ for all $\la > 0$.
\end{proof}

As a first application of the above expansion we show that the anomaly, when it exists, only depends on finitely many of the terms $\om(J,\la)$.

\begin{prop}\label{p:verhig}
Let $\dir_+ : \T{Dom}(\dir_+) \to \C H$ be a closed operator and let $A : \C H \to \C H$ be a bounded operator which together satisfy the conditions in Assumption \ref{a:unb} and Assumption \ref{a:tra} for some $n_0 \in \nn$. Let $n \in \{n_0,n_0 + 1,\ldots\}$.

The anomaly of $\C D_+ := \dir_+ + A$ in degree $n$ then exists if and only if the limit 
\[
\lim_{\la \to \infty} \la^n \cd \sum_{k=0}^{2n_0-1} \sum_{J \in \C B_k} \T{Tr}_s\big( \om(J,\la) \big)
\]
exists. In this case we have that
\[
\T{Ano}(\C D) = \lim_{\la \to \infty} \la^n \cd \sum_{k=0}^{2n_0-1} \sum_{J \in \C B_k} \T{Tr}_s\big( \om(J,\la) \big).
\]
\end{prop}
\begin{proof}
By Proposition \ref{p:degexp} it is enough to show that
\[
\lim_{\la \to \infty} \la^n \cd \sum_{k=2n_0}^\infty \sum_{J \in \C B_k}\T{Tr}_s\big( \om(J,\la) \big) = 0.
\]

By Lemma \ref{l:omeest} we may choose constants $C,K > 0$ (with $\ep = 3/4$) such that
\[
\la^n \|\om(J,\la)\|_1^s \leq \la^{n_0 - 1/4} \cd (\C B(J)+1) \cd (K/\sqrt{\la})^{\C B(J)}
\]
for all $\la \geq C$ and all $J \in \nn_0^n$.

Let now $\la \geq \T{max}\{2K^2,C\}$. It then follows from the above estimate that
\[
\begin{split}
\big\| \la^n \cd \sum_{k=2n_0}^\infty \sum_{J \in \C B_k} \om(J,\la) \big\|_1^s & \leq \sum_{k = 2n_0}^\infty \sum_{J \in \C B_k} (k+1) \cd
\big( K/\sqrt{\la} \big)^k \cd \la^{n_0 -1/4} \\
& \leq K^{2n_0} \cd \la^{-1/4} 
\cd \sum_{k = 2n_0}^\infty \sum_{J \in \C B_k} (1/2)^{k-2n_0}.
\end{split}
\]
But this proves the theorem. In fact we obtain 
\[
\Big| \la^n \cd \sum_{k=2n_0}^\infty \sum_{J \in \C B_k}\T{Tr}_s\big( \om(J,\la) \big) \Big| = O(\la^{-1/4})
\]
as $\la \to \infty$.
\end{proof}

\section{The anomaly for Dirac operators on $\rr^{2n}$}\label{s:anodir}
For this example we start with the flat Dirac operator on $\rr^{2n}$ defined as follows. Let $\cc_{2n-1}$ denote the Clifford algebra over $\rr^{2n-1}$. Recall that this is the unital $*$-algebra over $\cc$ with $(2n-1)$-generators $e_1,\ldots,e_{2n-1} \in \cc_{2n-1}$ satisfying the relations
\[
e_j = e_j^* \q e_i e_j + e_j e_i = 2 \de_{ij} \q \T{for all }\, i,j \in \{1,\ldots,2n-1\}.
\]
The unit will be denoted by $1 = e_0 \in \cc_{2n-1}$. We fix an irreducible representation $\pi_{2n-1} : \cc_{2n-1} \to \sL(\cc^{2^{n-1}})$. The flat Dirac operator on $\rr^{2n}$ is then given by
\begin{equation}\label{eq:dir}
\dir := \ma{cc}{0 & \dir_- \\ \dir_+ & 0} := \ma{cc}{0 & -\frac{\pa}{\pa x_{2n}} + i \sum_{j=1}^{2n-1} c_j \frac{\pa}{\pa x_j}
\\  \frac{\pa}{\pa x_{2n}} + i \sum_{j=1}^{2n-1} c_j \frac{\pa}{\pa x_j} & 0},
\end{equation}
where $c_j := \pi_{2n-1}(e_j)$, for all $j \in \{1,\ldots,2n-1\}$. The domain of $\dir$ is a direct sum of the Sobolev space $H^1(\rr^{2n}) \ot \cc^{2^{n-1}}$ with itself, thus $\T{Dom}(\dir_+) = \T{Dom}(\dir_-) = H^1(\rr^{2n}) \ot \cc^{2^{n-1}}$. The Clifford matrices act on the second component of the tensor product and the partial differentiation operators act on the first component. We recall that the Sobolev space $H^1(\rr^{2n})$ consists of the elements in $L^2(\rr^{2n})$ which have all weak derivatives of first order in $L^2(\rr^{2n})$.

We note that $\dir_+$ is normal and that the unbounded operator
\[
\dir_- \dir_+ = -\sum_{j=1}^{2n} \frac{\pa^2}{\pa x_j^2} = \dir_+ \dir_- : H^2(\rr^{2n}) \ot \cc^{2^{n-1}} \to L^2(\rr^{2n}) \ot \cc^{2^{n-1}}
\]
is a diagonal operator with the Laplacian on $\rr^{2n}$ in every entry on the diagonal. The notation $H^2(\rr^{2n})$ refers to the second Sobolev space on $\rr^{2n}$.

We let $C^\infty_b(\rr^{2n})$ denote the unital $*$-algebra of smooth complex valued functions with all derivatives bounded.

Let $N \in \nn$. Our strategy is to perturb the flat Dirac operator $\dir$ by a bounded operator corresponding to a $U(N)$-connection. More precisely, we define
\begin{equation}\label{eq:con}
\begin{split}
A : L^2(\rr^{2n}) \ot \cc^N \ot \cc^{2^{n-1}} \to L^2(\rr^{2n}) \ot \cc^N \ot \cc^{2^{n-1}} \q A := i a_{2n} - \sum_{j=1}^{2n-1} c_j a_j,
\end{split}
\end{equation}
where $a_1,\ldots,a_{2n} : \rr^{2n} \to M_N(\cc)$ are selfadjoint elements in $C^\infty_b\big(\rr^{2n}, M_N(\cc)\big)$ acting by multiplication on $L^2(\rr^{2n}) \ot \cc^N$. We are thus interested in the unbounded operator
\begin{equation}\label{eq:perdir}
\C D := \ma{cc}{
0 & \dir_- + A^* \\
\dir_+ + A & 0
} : H^1(\rr^{2n}) \ot \cc^{2^n} \ot \cc^N \to L^2(\rr^{2n}) \ot \cc^{2^n} \ot \cc^N,
\end{equation}
where
\[
\begin{split}
\dir_+ + A & = \big( \frac{\pa}{\pa x_{2n}} + i a_{2n} \big) + i \sum_{j=1}^{2n-1} c_j \big( \frac{\pa}{\pa x_j} + ia_j\big) \q \T{and} \\
\dir_- + A^* & = - \big( \frac{\pa}{\pa x_{2n}} + i a_{2n} \big) + i \sum_{j=1}^{2n-1} c_j \big( \frac{\pa}{\pa x_j} + ia_j\big)
\end{split}
\]

We remark that the bounded operators $A$ and $A^*$ preserve the domain of $\dir_+$ (the first Sobolev space) and that the sum of commutators is given by
\[
\begin{split}
[\dir_+,A^*] + [A,\dir_-] 
& = 2 \big[\frac{\pa}{\pa x_{2n}},-\sum_{j=1}^{2n-1} c_j a_j\big] + 2 \big[ i a_{2n},i \sum_{j=1}^{2n-1} c_j \frac{\pa}{\pa x_j} \big]  \\
& = 2 \sum_{j=1}^{2n-1} c_j \Big(  \frac{\pa a_{2n}}{\pa x_j} - \frac{\pa a_j}{\pa x_{2n}} \Big)
\end{split}
\]
This shows that the conditions in Assumption \ref{a:unb} are satisfied.

We note furthermore that 
\[
[A, A^*] = 2\big[i a_{2n}, -\sum_{j=1}^{2n-1} c_j a_j \big] = 2i \sum_{j=1}^{2n-1}c_j [a_j,a_{2n}].
\]
The bounded extension of the commutator $[\C D_+ ,\C D_-]$ is therefore given by the expression
\[
F := 2 \cd \sum_{j=1}^{2n-1} c_j \Big(  \frac{\pa a_{2n}}{\pa x_j} - \frac{\pa a_j}{\pa x_{2n}} + i [a_j,a_{2n}] \Big),
\]

The smoothness and boundedness conditions on the maps $a_1,\ldots,a_{2n} : \rr^{2n} \to M_N(\cc)$ imply that $A$ lies in $\T{OP}^0$ with respect to the Laplacian $\De : H^2(\rr^{2n}) \ot \cc^{2^{n-1}} \ot \cc^N \to L^2(\rr^{2n}) \ot \cc^{2^{n-1}} \ot \cc^N$. 

To deal with the trace class condition on $(1+\De)^{-i-1} F (1 + \De)^{-n +i}$ for $i \in \{-1,0,\ldots,n-1\}$ we make the following standing assumption:

\begin{assu}\label{a:fct}
Assume that the maps $a_1,\ldots,a_{2n-1} : \rr^{2n} \to M_N(\cc)$ are selfadjoint elements in $C_b^\infty\big(\rr^{2n},M_N(\cc)\big)$ such that the partial derivative
\[
\frac{\pa a_j}{\pa x_{2n}} : \rr^{2n} \to M_N(\cc)
\]
has compact support for each $j \in \{1,\ldots,2n-1\}$. Assume furthermore that $a_{2n} = 0$.
\end{assu}

The above assumption implies that 
\[
F = - 2 \cd \sum_{j=1}^{2n-1} c_j \frac{\pa a_j}{\pa x_{2n}} : \rr^{2n} \to M_{2^{n-1}}(\cc) \ot M_N(\cc)
\]
is smooth and compactly supported. It then follows from \cite[Theorem 4.1]{Sim:TIA} and \cite[Theorem 4.5]{Sim:TIA} that $(1 + \De)^{-i-1} F (1 + \De)^{-n+i + \ep}$ is of trace class for all $i \in \{-1,0,\ldots,n-1\}$ and all $\ep \in (0,1)$.

We recollect what we have proved thus far in the following:

\begin{thm}\label{t:exihom}
Suppose that the maps $a_1,\ldots,a_{2n-1},a_{2n} : \rr^{2n} \to M_N(\cc)$ satisfy the conditions in Assumption \ref{a:fct}. Then the pair consisting of the closed unbounded operator $\dir_+$ and the bounded operator $A$ as defined in \eqref{eq:dir} and \eqref{eq:con} satisfies the conditions in Assumption \ref{a:unb} and Assumption \ref{a:tra} for $n_0 = n$. In particular, there is a well-defined homological index in degree $m$, $\T{H-Ind}_m(T_\la)$ for each $\la \in (0,\infty)$ and each $m \geq n$.
\end{thm}

The unbounded operators $V_1 \T{ and } V_2 : \T{Dom}(\dir_+) \to L^2(\rr^{2n}) \ot \cc^{2^{n-1}}$ turn out to be first order differential operators. Indeed, we have that
\[
\begin{split}
V_1 & = A \dir_- + \dir_+ A^* + A A^* \\
& = -i \sum_{j,k=1}^{2n-1} c_j c_k a_j \frac{\pa}{\pa x_k} - i \sum_{j,k=1}^{2n-1} c_k c_j \frac{\pa}{\pa x_k} a_j
- \big[ \frac{\pa}{\pa x_{2n}}, \sum_{j=1}^{2n-1} c_j a_j \big] + \sum_{j,k = 1}^{2n-1} c_j c_k a_j a_k \\
& = i \sum_{1 \leq j < k \leq 2n-1}  c_j c_k \big( \frac{\pa a_j}{\pa x_k} - \frac{\pa a_k}{\pa x_j} -i [a_j,a_k]\big) -i \cd \sum_{j= 1}^{2n-1} \big( a_j \frac{\pa}{\pa x_j} + \frac{\pa}{\pa x_j} a_j + i a_j^2\big) + \frac{1}{2} F 
\end{split}
\]
And it then follows from the identity $F(\xi) = (V_1 - V_2)(\xi)$ for all $\xi \in \T{Dom}(\dir_-)$ that
\[
V_2 = -i \sum_{j= 1}^{2n-1} \big( a_j \frac{\pa}{\pa x_j} + \frac{\pa}{\pa x_j} a_j + i a_j^2\big)
- \frac{1}{2} F + i \sum_{1 \leq j < k \leq 2n-1}  c_j c_k \big( \frac{\pa a_j}{\pa x_k} - \frac{\pa a_k}{\pa x_j} -i [a_j,a_k] \big).
\]

\begin{notation}\label{n:difope}
Define the first order differential operator
\[
S := -i  \sum_{j= 1}^{2n-1} \big( 2 a_j \frac{\pa}{\pa x_j} + \frac{\pa a_j}{\pa x_j} + i a_j^2\big) : H^1(\rr^{2n}) \ot \cc^{N \cd 2^{n-1}} \to L^2(\rr^{2n}) \ot \cc^{N \cd 2^{n-1}}.
\]
Define the multiplication operators
\[
\begin{split}
F & := - 2 \cd \sum_{j=1}^{2n-1} c_j \frac{\pa a_j}{\pa x_{2n}} : L^2(\rr^{2n}) \ot \cc^{N \cd 2^{n-1}} \to L^2(\rr^{2n}) \ot \cc^{N \cd 2^{n-1}} \q \T{and} \\
G & := i \sum_{1 \leq j < k \leq 2n-1}  c_j c_k \big( \frac{\pa a_j}{\pa x_k} - \frac{\pa a_k}{\pa x_j} -i [a_j,a_k] \big) \\ 
& \qqqq : L^2(\rr^{2n}) \ot \cc^{N \cd 2^{n-1}} \to L^2(\rr^{2n}) \ot \cc^{N \cd 2^{n-1}}.
\end{split}
\]
\end{notation}

Using the above notation, we have that $V_1 = S + F/2 + G$ and $V_2 = S - F/2 + G$.

It is important, before continuing, to relate the bounded operators $F$ and $G$ to the $U(N)$-connection used for constructing the perturbation of the flat Dirac operator $\dir$. 

Since $a_{2n} = 0$, the connection $1$-form is given by
\[
\om := i \sum_{j=1}^{2n-1} a_j dx_j \in \Om^1\big(\rr^{2n},M_N(\cc)\big).
\]
And it follows that the curvature operator is induced by the $2$-form
\[
\begin{split}
R := \om \we \om + d\om & = i \sum_{1 \leq j < k \leq 2n-1} \Big( \frac{\pa a_k}{\pa x_j} - \frac{\pa a_j}{\pa x_k} + i [a_j,a_k] \Big) dx_j \we dx_k  \\
& \qqq -i \sum_{j = 1}^{2n-1} \frac{\pa a_j}{\pa x_{2n}} dx_j \we dx_{2n}.
\end{split}
\]
Let $\pi_{2n} : \cc_{2n} \to \sL(\cc^{2^n})$ denote the representation of the Clifford algebra with $(2n)$-generators $e_1,\ldots,e_{2n}$ used for the definition of our Dirac-type operator $\C D$. The Clifford contraction of the curvature form can then be written as
\begin{equation}\label{eq:confg}
\begin{split}
\G K & := \sum_{1 \leq j < k \leq 2n} \pi_{2n}(e_j e_k) R\big(\frac{\pa}{\pa x_j},\frac{\pa}{\pa x_k}\big) \\
& = i \sum_{1 \leq j < k \leq 2n-1} \pi_{2n}(e_j e_k) \Big( \frac{\pa a_k}{\pa x_j} - \frac{\pa a_j}{\pa x_k} + i [a_j,a_k] \Big) \\
& \qqq -i \sum_{j = 1}^{2n-1} \pi_{2n}(e_j e_{2n}) \frac{\pa a_j}{\pa x_{2n}} \\
& = \ma{cc}{
F/2 - G & 0 \\
0 & -F/2 - G
}.
\end{split}
\end{equation}
We may thus rewrite the differential operator $V := \ma{cc}{V_2 & 0 \\ 0 & V_1}$ as $V = S - \G K$.

{\it
The goal of the next Section is to provide a local formula for the anomaly (in degree $m \geq n$) of $\C D_+ = \dir_+ + A$ in terms of the above Clifford contraction $\G K$}.

\section{A local formula for the anomaly}\label{s:locano}
Throughout this section we will operate in the context of Section \ref{s:anodir}. The conditions in Assumption \ref{a:fct} on the $U(N)$-connection will in particular be in effect. We are then aiming to compute the anomaly of the perturbed Dirac operator $\C D_+ = \dir_+ + A$ as the integral of a $(2n)$-form associated to the Clifford contraction of the curvature form, that is, in terms of a polynomial in the operators $F$ and $G$.

Our first concern is to eliminate a substantial amount of terms in the general resolvent expansion of the homological index provided in Proposition \ref{p:degexp}. To do this we start with a short preliminary on Clifford algebras.

\subsection{Preliminaries on Clifford algebras}\label{ss:precli}
We define irreducible representations of the Clifford algebra $\cc_k$ for all $k \in \nn$ recursively as follows. For $k = 1$, let
\[
\pi_1 : \cc_1 \to \sL(\cc) \q \pi : e_1 \mapsto 1.
\]
For $k = 2m$, let $\pi_{2m} : \cc_{2m} \to \sL(\cc^{2^m})$ be defined by
\[
\pi_{2m} : e_i \mapsto \fork{ccc}{
\pi_{2m-1}(e_i) \ot \ma{cc}{0 & 1 \\ 1 & 0} & \T{for} & i \in \{1,\ldots,2m-1\} \\
1 \ot \ma{cc}{0 & i \\ -i & 0} & \T{for} & i = 2m.
}
\]
For $k = 2m+1 > 1$, let $\pi_{2m+1} : \cc_{2m+1} \to \sL(\cc^{2^m})$ be defined by
\[
\pi_{2m+1} : e_i \mapsto \fork{ccc}{
\pi_{2m}(e_i)  & \T{for} & i \in \{1,\ldots,2m\} \\
\ma{cc}{
1_{2^{m-1}} & 0 \\
0 & - 1_{2^{m-1}}
} & \T{for} & i = 2m+1.
}
\]

\begin{notation}
Let $k \in \nn$. For each subset $I = \{i_1,\ldots,i_j\} \su \{1,\ldots,k\}$ with $i_1 < \ldots < i_j$, define
\[
e_I := e_{i_1} \clc e_{i_j} \in \cc_k \q e_{\emptyset} := 1 \in \cc_k.
\]
The odd part of $\cc_k$ is the vector subspace $\cc_{\T{odd}} \su \cc_k$ defined by
\[
\cc_{\T{odd}} := \T{span}_{\cc}\big\{ e_{i_1} \clc e_{i_j} \, | \, 1 \leq i_1 < \ldots < i_j \leq k, \, j \, \T{ odd}\big\}.
\]
The even part of $\cc_k$ is the vector subspace $\cc_{\T{ev}} \su \cc_k$ defined by
\[
\cc_{\T{ev}} := \T{span}_{\cc}\big\{ e_{i_1} \clc e_{i_j} \, | \, 1 \leq i_1 < \ldots < i_j \leq k, \, j \, \T{ even}\big\}.
\]
\end{notation}

\begin{lemma}\label{l:tracli}
Let $m \in \nn$. Let $I \su \{1,\ldots,2m-1\}$, then
\[
\T{Tr}\big( \pi_{2m-1}(e_I) \big) = \fork{ccc}{
2^{m-1} & \T{for} & I = \emptyset \\
(-2i)^{m-1} & \T{for} & I = \{1,\ldots,2m-1\} \\
0 & \T{for} & I \neq \emptyset \, , \, \{1,\ldots,2m-1\}.
}
\]
\end{lemma}
\begin{proof}
The fact that $\T{Tr}(\pi_{2m-1}(1)) = 2^{m-1}$ is obvious. We may thus suppose that $I \neq \emptyset$. 

Suppose that the number of elements $j$ in $I = \{i_1,\ldots,i_j\}$ is even. Then
\[
\T{Tr}\big( \pi_{2m-1}(e_I) \big) = - \T{Tr}\big( \pi_{2m-1}(e_{i_j} \cd e_{i_1} \clc e_{i_{j-1}}) \big)
= - \T{Tr}\big( \pi_{2m-1}(e_I) \big).
\]
This shows that $\T{Tr}\big( \pi_{2m-1}(e_I) \big) = 0$ in this case.

We may thus suppose that the number of elements $j$ in $I$ is odd.

Suppose that $j < 2m-1$. Choose an element $k \in \{1,\ldots,2m-1\}\sem J$. The endomorphism $\pi_{2m-1}(e_I) : \cc^{2^{m-1}} \to \cc^{2^{m-1}}$ then anticommutes with the selfadjoint unitary operator $\pi_{2m-1}(e_k)$. This implies that $\T{Tr}\big( \pi_{2m-1}(e_I) \big) = 0$.

We may thus suppose that $I = \{1,\ldots,2m-1\}$. The identity $\T{Tr}\big( \pi_{2m-1}(e_I)\big) = (-2i)^{m-1}$ then follows by induction. Indeed, for $m > 1$ we have that
\[
\begin{split}
\pi_{2m-1}(e_I) & = \pi_{2m-3}(e_1 \clc e_{2m-3}) \ot \ma{cc}{0 & 1 \\ 1 & 0} \cd \ma{cc}{0 & i \\ -i & 0} \cd \ma{cc}{1 & 0 \\ 0 & - 1} \\
& = \pi_{2m-3}(e_1 \clc e_{2m-3}) \ot \ma{cc}{-i & 0 \\ 0 & -i},
\end{split}
\]
which implies that $\T{Tr}\big( \pi_{2m-1}(e_I) \big) = (-2i) \cd \pi_{2m-3}(e_1 \clc e_{2m-3})$.
\end{proof}

\subsection{Vanishing of lower Clifford terms}
Let $m \in \{n,n+1,\ldots\}$. Recall from Proposition \ref{p:degexp} that the homological index of $T_\la := \la^{-1/2} \C D_+ (1+ \la^{-1} \C D_- \C D_+)$ in degree $m$ can be written as
\[
\T{H-Ind}_m(T_\la) = \la^m \cd \sum_{k=0}^\infty \sum_{J \in \C B_k} \T{Tr}_s\big( \om(J,\la) \big),
\]
where the sum converges absolutely whenever $\la \geq C$ for an appropriate constant $C > 0$. The index sets are given by $\C B_k := \big\{ J = (j_1,\ldots,j_m) \in \nn_0^m \, | \, j_1 \plp j_m = k\big\}$ for each $k \in \nn_0$ whereas the bounded operator $\om(J,\la) = \ma{cc}{\om_2(J,\la) & 0 \\ 0 & \om_1(J,\la)}$ is defined in Notation \ref{n:omedeg} for each $J \in \nn_0^m$ and each $\la > 0$.

Since the irreducible representation $\pi_{2n} : \cc_{2n} \to M_{2^n}(\cc)$ is an isomorphism we have a vector space decomposition $M_{2^n}(\cc) \cong \bop_{I \su \{1,\ldots,2n\}} \cc \pi_{2n}(e_I)$. In particular we have an idempotent
\begin{equation}\label{eq:cliide}
E_I : \sL(L^2(\rr^{2n})) \ot M_N(\cc) \ot M_{2^n}(\cc) \to \sL(L^2(\rr^{2n})) \ot M_N(\cc) \ot M_{2^n}(\cc) 
\end{equation}
with image, $\T{Im}(E_I) = \sL(L^2(\rr^{2n})) \ot M_N(\cc) \ot \cc \pi_{2n}(e_I)$, for each $I \su \{1,\ldots,2n\}$.

We shall see in this section that the "lower order terms" do not contribute to the homological index. More precisely, we will prove that the super trace of the terms $\om(J,\la)$ is trivial whenever $J \in \C B_k$ for some $k \in \{0,\ldots,n-1\}$ and $\la > 0$.

\begin{lemma}\label{l:vollef}
Let $J \in \nn_0^m$ and let $\la > 0$. Then
\[
\T{Tr}_s\big(\om(J,\la)\big) = \T{Tr}_s\big( E_{\{1,\ldots,2n\}}(\om(J,\la)) \big).
\]
\end{lemma}
\begin{proof}
Let us define the super trace $\T{Tr}_s : M_{2^n}(\cc) \to \cc$, $\T{Tr}_s : T \mapsto \T{Tr}\big( T \cd \pi_{2n+1}(e_{2n+1}) \big)$, where $\T{Tr} : M_{2^n}(\cc) \to \cc$ denotes the usual matrix trace.

We may then rewrite the super trace $\T{Tr}_s : \sL^1_s\big( L^2(\rr^{2n}) \ot \cc^N \ot \cc^{2^n} \big) \to \cc$ as the composition $\T{Tr} \ci (1 \ot 1 \ot \T{Tr}_s)$, where $\T{Tr} : \sL^1\big( L^2(\rr^{2n}) \ot \cc^N \big) \to \cc$ denotes the usual operator trace.

An application of Lemma \ref{l:tracli} now proves the lemma.
\end{proof}

Before we continue, we make a short digression on the form of the bounded operator $\om(J,\la)$, for $J \in \nn_0^m$ and $\la > 0$. Recall from the discussion after Notation \ref{n:difope} that the first order differential operator $V := \ma{cc}{V_2 & 0 \\ 0 & V_1}$ can be written as $V = S - \G K$, where $\G K = \ma{cc}{F/2 - G & 0 \\ 0 & -F/2 - G}$ denotes the Clifford contraction of the curvature form associated with our fixed $U(N)$-connection. 

Notice next that the Clifford algebra $\cc_{2n}$ is a filtered algebra $\{0\} \su C_0 \su \ldots \su C_{2n-1} \su C_{2n} = \cc_{2n}$. The subspace $C_i \su \cc_{2n}$ is defined by
\begin{equation}\label{eq:filcli}
C_i := \T{span}_{\cc}\big\{ e_I \, | \, I \su \{1,\ldots,2n\} \, , \, |I| \leq i\big\}
\end{equation}
for each $i \in \{0,\ldots,2n\}$, where $|I|$ denotes the number of elements of a subset $I \su \{1,\ldots,2n\}$. 

Let now $\la > 0$ and let us apply the notation 
\[
V_\la := V \cd (\la + \De)^{-1} \q S_\la := S \cd (\la + \De)^{-1} \q \G K_\la := \G K \cd (\la + \De)^{-1}.
\]
The bounded operator $S_\la$ then lies in the subspace $\sL\big( L^2(\rr^{2n}) \ot \cc^N\big) \ot \cc 1_{2^n}$ and the bounded operator $\G K_\la$ lies in the subspace $\sL\big( L^2(\rr^{2n}) \ot \cc^N\big) \ot \pi_{2n}(C_2)$ of $\sL\big( L^2(\rr^{2n}) \ot \cc^{N\cd 2^n} \big)$.

Let $J = (j_1,\ldots,j_m) \in \nn_0^m$. The above discussion then shows that
\[
\om(J,\la) = (\la + \De)^{-1} V_\la^{j_1} \clc (\la + \De)^{-1} V_\la^{j_m} \in \sL\big( L^2(\rr^{2n}) \ot \cc^N\big) \ot \pi_{2n}(C_{2\cd \C B(J)}),
\]
where we put $C_i := \cc_{2n}$ whenever $i \geq 2n$. Recall here that $\C B(J) = j_1 \plp j_m$.

A combination of these observations and Lemma \ref{l:vollef} implies the following result. It provides a significant reduction of terms in the resolvent expansion for the homological index.

\begin{prop}\label{p:low}
Suppose that the maps $a_1,\ldots,a_{2n} : \rr^{2n} \to M_N(\cc)$ satisfy the conditions in Assumption \ref{a:fct}. Then
\[
\T{Tr}_s\big( \om(J,\la) \big) = 0,
\]
for all $\la > 0$  and all $J = (j_1,\ldots,j_m) \in \nn_0^m$ with $j_1 \plp j_m \leq n-1$. 
\end{prop}

\subsection{Vanishing of higher order terms}
In this subsection we will continue reducing the number of terms which contribute to the anomaly of the Dirac-type operator $\C D$. We are thus interested in comparing the super trace of the operators $\la^m \cd \om(J,\la)$ with the size of the scaling parameter $\la > 0$ whenever $J \in \C B_k := \C B^{-1}(\{k\})$ for some $k \in \{n + 1,\ldots,2n - 1\}$. 

Let $\C H := L^2(\rr^{2n}) \ot \cc^N \ot \cc^{2^{n-1}}$ and recall that $\sL^1_s(\C H) \su \sL(\C H \op \C H)$ denotes the Banach space of bounded operators
\[
T = \ma{cc}{
T_{11} & T_{12} \\
T_{21} & T_{22}
}
\]
with $T_{11} -T_{22} \in \sL^1(\C H)$. The norm is given by $\|T\|_1^s := \|T\| + \|T_{11} - T_{22}\|_1$.

We will apply the notation $P_+ : \C H \op \C H \to \C H$ and $P_- : \C H \op \C H \to \C H$ for the orthogonal projections onto the first and the second component in $\C H \op \C H$, respectively.

Notice that the terms of the form $\la^m \om(J,\la)$ converges to zero in the Banach space $\sL^1_s(\C H)$ as $\la \to \infty$ whenever $J \in \C B_k := \C B^{-1}(\{k\})$ for some $k \in \{2n,2n+1,\ldots\}$. This is a consequence of Lemma \ref{l:omeest}.

In order to obtain the desired estimates on the terms $\la^m \cd \om(J,\la)$ for $J \in \C B_k$ for $k \in \{n+1,\ldots,2n-1\}$ we need to consider a more detailed expansion of the bounded operators $\om(J,\la)$. To this end we introduce the following:

\begin{notation}
Let $j \in \nn_0$. For any sequence $L = (l_1,\ldots,l_j) \in \{0,1\}^j$, define the bounded operator
\[
\ga_\la(L) := \ga_\la(l_1) \clc \ga_\la(l_j) \q \ga_\la(0) = S \cd (\la + \De)^{-1} \, , \, \ga_\la(1) = - \G K \cd (\la + \De)^{-1}.
\]
When $j = 0$, we identify the set $\{0,1\}^j$ with a basepoint $*$ and put $\ga_\la(*) := 1$.

For any sequence $L = (l_1,\ldots,l_j) \in \{0,1\}^j$, let $D(L) = l_1 \plp l_j$. When $j = 0$, we put $D(*) = 0$.
\end{notation}

For any $J = (j_1,\ldots,j_m) \in \nn_0^m$ and any $\la > 0$ we then have that
\begin{equation}\label{eq:omeexp}
\begin{split}
\om(J,\la) & = (\la + \De)^{-1} \big(\ga_\la(0) + \ga_\la(1) \big)^{j_1} \clc (\la + \De)^{-1} \big(\ga_\la(0) + \ga_\la(1) \big)^{j_m} \\
& = (\la + \De)^{-1} \Big( \sum_{L_1 \in \{0,1\}^{j_1}} \ga_\la(L_1) \Big) \clc (\la + \De)^{-1} \Big( \sum_{L_m \in \{0,1\}^{j_m}} \ga_\la(L_m) \Big).
\end{split}
\end{equation}

In the next lemmas we will therefore prove a few estimates on operators of the type $\ga_\la(L)$.

\begin{lemma}\label{l:actfin}
Let $i \in \zz$, let $j \in \nn_0$ and let $L \in \{0,1\}^j$. Then
\[
\big\| \si_\la^i\big( \ga_\la(L) \big) \big\| = O(\la^{-(D(L) + j)/2})
\]
as $\la \to \infty$.
\end{lemma}
\begin{proof}
Since $\|(\la + \De)^{-1/2}\| = O(\la^{-1/2})$ as $\la \to \infty$, it is enough to show that
\[
\big\| \si_\la^i\big(S \cd (\la + \De)^{-1/2}\big) \big\| = O(1) \q \T{and} \q \big\|\si_\la^i(\G K) \big\| = O(1)
\]
as $\la \to \infty$. But this follows easily from Lemma \ref{l:disanaest}.
\end{proof}

\begin{lemma}\label{l:galest}
Let $j \in \nn_0$ and let $L = (l_1,\ldots,l_j) \in \{0,1\}^j$. Then 
\[
(\la + \De)^{-i} \ga_\la(L) (\la + \De)^{-m+i} \in \sL^1_s(\C H )
\]
for all $i \in \{1,\ldots,m\}$ and we have the norm estimate
\[
\big\|(\la + \De)^{-i} \ga_\la(L) (\la + \De)^{-m+i}\big\|_1^s = O(\la^{n-m+1/4 - (D(L) + j)/2})
\]
as $\la \to \infty$.
\end{lemma}
\begin{proof}
Let $i \in \{1,\ldots,m\}$. 

Notice first that 
\[
\big\| (\la + \De)^{-i} \ga_\la(L) (\la + \De)^{-m+i} \big\| = O(\la^{-m - (D(L)+j)/2})
\]
as $\la \to \infty$ by an application of Lemma \ref{l:actfin}.

Suppose that $D(L) = 0$. We then have that 
\[
(\la + \De)^{-i} \ga_\la(L) (\la + \De)^{-m+i} = (\la + \De)^{-i} \cd S_\la^j \cd (\la + \De)^{-m+i}
\]
This shows that $(\la + \De)^{-i} \ga_\la(L) (\la + \De)^{-m+i} \in \sL^1_s(\C H)$ with
\[
\begin{split}
\big\| (\la + \De)^{-i} \ga_\la(L) (\la + \De)^{-m+i} \big\|_1^s 
& = \big\| (\la + \De)^{-i} \ga_\la(L) (\la + \De)^{-m+i} \big\| \\
& = O(\la^{-m - (D(L)+j)/2})
\end{split}
\]
as $\la \to \infty$.

Suppose that $D(L) \geq 1$. Choose an $r \in \{1,\ldots,j\}$ such that $L_i = (l_1,\ldots,l_{r-1},1,l_{r+1},\ldots,l_j)$. We then have that
\[
\begin{split}
& (\la + \De)^{-i} \ga_\la(L) (\la + \De)^{-m+i} \\
& \q = - \si^{-i}_\la\big( \ga_\la(l_1) \clc \ga_\la(l_{r-1}) \big) \cd (\la + \De)^{-i} \G K (\la + \De)^{-m-1 + i} \\
& \qqq \cd \si^{m-i}_\la \big( \ga_\la(l_{r+1}) \clc \ga_\la(l_j) \big).
\end{split}
\]
This shows that $(\la + \De)^{-i} \ga_\la(L) (\la + \De)^{-m+i} \in \sL^1_s(\C H)$. To prove the desired norm estimate, we recall from \eqref{eq:tranorest0} (with $\ep = 3/4$) that
\[
\big\| (\la + \De)^{-i} \G K (\la + \De)^{-m-1 + i} \big\|_1^s = O(\la^{n-m-3/4})
\]
as $\la \to \infty$. An application of Lemma \ref{l:actfin} now yields that
\[
\big\| (\la + \De)^{-i} \ga_\la(L) (\la + \De)^{-m+i} \big\|_1^s = O(\la^{n-m+1/4 - (D(L)+j)/2}).
\]

This proves the lemma.
\end{proof}

\begin{lemma}\label{l:galestfin}
Let $J = (j_1,\ldots,j_m) \in \nn_0^m$ and let $L_1 \in \{0,1\}^{j_1}, \ldots, L_m \in \{0,1\}^{j_m}$. Then
\[
(\la + \De)^{-1} \ga_\la(L_1) \clc (\la + \De)^{-1} \ga_\la(L_m) \in \sL^1_s(\C H)
\]
and we have the norm estimate
\[
\big\| (\la + \De)^{-1} \ga_\la(L_1) \clc (\la + \De)^{-1} \ga_\la(L_m) \big\|_1^s = O\big(\la^{n-m+1/4 - (D(L) + \C B(J))/2}\big)
\]
as $\la \to \infty$, where $D(L) := D(L_1) \plp D(L_m)$ and $\C B(J) := j_1 \plp j_m$.
\end{lemma}
\begin{proof}
The proof follows the same pattern as the proof of Lemma \ref{l:galest}. It is therefore left to the reader.
\end{proof}

Recall now that $\cc_{2n}$ is a filtered algebra with filtration $\{0\} \su C_0 \su \ldots \su C_{2n-1} \su C_{2n} = \cc_{2n}$ given by the subspaces defined in \eqref{eq:filcli}. Notice that the Clifford contraction $\G K$ lies in $\sL\big(L^2(\rr^{2n}) \ot \cc^N\big) \ot \pi_{2n}(C_2)$ and that the bounded operator $S \cd (\la + \De)^{-1}$ lies in $\sL\big(L^2(\rr^{2n}) \ot \cc^N\big) \ot \pi_{2n}(C_0)$.

Let $j \in \nn_0$ and let $L = (l_1,\ldots,l_j) \in \{0,1\}^j$. It follows from the above observation that
\[
\ga_\la(L) \in \sL\big(L^2(\rr^{2n}) \ot \cc^N\big) \ot \pi_{2n}(C_{2 \cd D(L)}),
\]
where $D(L) := l_1 \plp l_j$.

We are now ready to prove the main result of this section.

\begin{prop}\label{p:hig}
Suppose that the maps $a_1,\ldots,a_{2n} : \rr^{2n} \to M_N(\cc)$ satisfy the conditions in Assumption \ref{a:fct}. Let $k \in \{n+1,\ldots,2n-1\}$ and let $J = (j_1,\ldots,j_m) \in \C B_k$. Then
\[
\lim_{\la \to \infty} \la^m \cd \T{Tr}_s\big( \om(J,\la) \big) = 0.
\]
\end{prop}
\begin{proof}
Let $L_1 \in \{0,1\}^{j_1}, \ldots, L_m \in \{0,1\}^{j_m}$.

By Lemma \ref{l:galestfin} and the identity in \eqref{eq:omeexp} it is enough to show that
\[
\lim_{\la \to \infty} \la^m \T{Tr}_s\big( (\la + \De)^{-1} \ga_\la(L_1) \clc (\la + \De)^{-1} \ga_\la(L_m) \big) = 0.
\]

As in Lemma \ref{l:vollef} we have that
\[
\begin{split}
& \T{Tr}_s\big( (\la + \De)^{-1} \ga_\la(L_1) \clc (\la + \De)^{-1} \ga_\la(L_m) \big) \\
& \q = \T{Tr}_s\Big( E_{\{1,\ldots,2n\}} \big( (\la + \De)^{-1} \ga_\la(L_1) \clc (\la + \De)^{-1} \ga_\la(L_m) \big)\Big),
\end{split}
\]
where $E_{\{1,\ldots,2n\}} : \sL\big(L^2(\rr^{2n}) \ot \cc^{N \cd 2^n}\big) \to \sL\big(L^2(\rr^{2n}) \ot \cc^{N \cd 2^n}\big)$ is the idempotent associated with the vector space decomposition $\cc_{2n} \cong \cc e_{\{1,\ldots,2n\}} \op C_{2n-1}$ of the Clifford algebra. 

Notice that the operator $(\la + \De)^{-1} \ga_\la(L_1) \clc (\la + \De)^{-1} \ga_\la(L_m)$ lies in $\sL\big(L^2(\rr^{2n}) \ot \cc^N\big) \ot \pi_{2n}(C_{2 \cd D(L)})$. We may thus suppose that $D(L) = D(L_1) \plp D(L_m) \geq n$.

By Lemma \ref{l:galestfin} we can choose constants $C,K > 0$ such that
\[
\begin{split}
\big| \la^m \T{Tr}_s\big( (\la + \De)^{-1} \ga_\la(L_1) \clc (\la + \De)^{-1} \ga_\la(L_m) \big) \big| 
& \leq K \cd \la^{n + 1/4 - (D(L) + \C B(J))/2} \\
& \leq K \cd \la^{-1/4},
\end{split}
\]
for all $\la \geq C$, where we have used that $D(L) \geq n$ and $\C B(J) \geq n+1$.

This proves the proposition.
\end{proof}

\subsection{A local formula for the anomaly}
Let $m \in \{n,n+1,\ldots\}$. Throughout this section the conditions in Assumption \ref{a:fct} on the Dirac type operator $\C D$ introduced in Section \ref{s:anodir} will be in effect. Let us start by recollecting what we have proved thus far. It follows from Proposition \ref{p:verhig}, Proposition \ref{p:hig} and Proposition \ref{p:low} that the anomaly in degree $m$ of the Dirac type operator $\C D$ exists if and only if the limit
\[
\lim_{\la \to \infty} \la^m \cd \sum_{J \in \C B_n} \T{Tr}_s\big( \om(J,\la) \big)
\]
exists and in this case, we have that
\[
\T{Ano}_m(\C D) = \lim_{\la \to \infty} \la^m \cd \sum_{J \in \C B_n} \T{Tr}_s\big( \om(J,\la) \big).
\]
Recall here that 
\[
\om(J,\la) := \ma{cc}{\om_2(J,\la) & 0 \\ 0 & \om_1(J,\la)} \in \sL(\C H)
\]
where the bounded operators $\om_2(J,\la)$ and $\om_1(J,\la)$ have been defined in Notation \ref{n:omedeg} for all $J \in \nn_0^m$ and all $\la > 0$. The index set $\C B_n$ consists of all $J \in \nn_0^m$ with $j_1 \plp j_m = n$.

\begin{lemma}\label{l:onldeg}
The anomaly of $\C D$ in degree $m$ exists if and only if the limit
\[
\lim_{\la \to \infty} \la^m \cd \T{Tr}_s\big( \G K^n \cd (\la + \De)^{-n-m} \big)
\]
exists and in this case
\[
\T{Ano}_m(\C D) = {m + n-1 \choose m-1} \cd (-1)^n \cd \lim_{\la \to \infty} \la^m \cd \T{Tr}_s\big( \G K^n \cd (\la + \De)^{-n-m} \big).
\]
\end{lemma}
\begin{proof}
Let $J = (j_1,\ldots,j_m) \in \C B_n$ and let $\la > 0$. Recall from Lemma \ref{l:vollef} that 
\[
\T{Tr}_s\big( \om(J,\la) \big) = \T{Tr}_s\big( E_{\{1,\ldots,2n\}}\big( \om(J,\la) \big) \big),
\]
where the idempotent $E_{\{1,\ldots,2n\}} : \sL(\C H) \to \sL(\C H)$ was introduced in \eqref{eq:cliide}.

Notice also that $\G K_\la := \G K (\la + \De)^{-1} \in \sL\big(L^2(\rr^{2n}) \ot \cc^N\big) \ot \pi_{2n}(C_2)$ and $S_\la := S (\la + \De)^{-1} \in \sL\big(L^2(\rr^{2n}) \ot \cc^N\big) \ot \pi_{2n}(C_0)$, where the subspaces $C_i \su \cc_{2n}$, $i \in \{0,\ldots,2n\}$ were introduced in \eqref{eq:filcli}.

We therefore have that
\begin{equation}\label{eq:idedegn}
\begin{split}
& E_{\{1,\ldots,2n\}}\big( \om(J,\la)\big) 
= (-1)^n \cd E_{\{1,\ldots,2n\}}\big( (\la + \De)^{-1}\G K_\la^{j_1} \clc (\la + \De)^{-1} \G K_\la^{j_m} \big) \\
& \q = (-1)^n \cd E_{\{1,\ldots,2n\}}\Big( \si_\la^{-1}(\G K) \clc \si_\la^{-j_1}(\G K) \cd \si_\la^{-j_1 -2}(\G K) \clc \si_\la^{-j_1 - j_2 -1}(\G K) \\
& \qq \clc \si_\la^{-j_1-\ldots - j_{m-1} -m}(\G K) \clc \si_\la^{-j_1 -\ldots - j_m -m+1}(\G K)
 \cd (\la + \De)^{-n-m} \Big).
\end{split}
\end{equation}

Recall now from \eqref{eq:tranorest} (with $\ep = 3/4$) that $\big\| (\la + \De)^{-n-m+1} \G K (\la + \De)^{-1}\big\|_1^s = O(\la^{1/4-m})$ as $\la \to \infty$. It therefore follows from \eqref{eq:idedegn} and Lemma \ref{l:difact} that
\begin{equation}\label{eq:remact}
\la^m \cd \big\| (-1)^n E_{\{1,\ldots,2n\}}\big( (\la + \De)^{-n-m+1} \G K^n (\la + \De)^{-1}\big) - E_{\{1,\ldots,2n\}}\big( \om(J,\la) \big) \big\|_1^s = O(\la^{-1/4})
\end{equation}
as $\la \to \infty$. An application of the cyclic property of the operator trace then yields that
\[
\lim_{\la \to \infty} \la^m \Big( \T{Tr}_s\big( \om(J,\la)\big) - (-1)^n \T{Tr}_s\big( \G K^n (\la + \De)^{-n-m} \big)\Big) = 0.
\]

The discussion in the beginning of this section now entails the result of this lemma since the number of elements in the index set $\C B_n \su \nn_0^m$ is ${m+n-1 \choose m-1}$.
\end{proof}

We will consider the matrix valued $(2n)$-form $\sA \in \Om^{2n}(\rr^{2n},M_N(\cc))$ given by
\begin{equation}\label{eq:fordef}
\sA : x \mapsto \si_{2n}(x)(\G K^n(x))  \q \T{for all } \, x \in \rr^{2n}
\end{equation}
where $\si_{2n}(x) : M_{2^n}(\cc) \ot M_N(\cc) \to \La^{2n}_x(\rr^{2n}) \ot M_N(\cc)$ is defined by 
\[
\si_{2n}(x)(\pi_{2n}(e_I) \ot M) := \fork{ccc}{
0 & \T{for} & I \neq \{1,\ldots,2n\} \\
(dx_1 \wlw dx_{2n})(x) \ot M & \T{for} & I = \{1,\ldots,2n\}.
}
\]
for all $x \in \rr^{2n}$ and all $M \in M_N(\cc)$. Recall here that $\G K : \rr^{2n} \to M_N(\cc) \ot M_{2^n}(\cc)$ is the \emph{Clifford contraction} of the curvature form associated with the $U(N)$-connection which determines our Dirac type operator, see the discussion in the end of Section \ref{s:anodir}. The identity in \eqref{eq:confg} shows that this Clifford contraction can be expressed in a simple way in terms of the bounded operators $F$ and $G$ which we introduced in Notation \ref{n:difope}.

The following theorem is the main result of this paper. We refer to the beginning of Section \ref{s:anodir} for a definition of the Dirac type operator $\C D : H^1(\rr^{2n}) \ot \cc^{2^n \cd N} \to L^2(\rr^{2n}) \ot \cc^{2^n \cd N}$.

\begin{theorem}\label{t:locano}
Let $a_1,\ldots,a_{2n} : \rr^{2n} \to M_N(\cc)$ be maps which satisfy the conditions in Assumption \ref{a:fct}. Let $m \in \{n,n+1,\ldots\}$. Then the anomaly in degree $m$ associated with the Dirac type operator $\C D := \dir + \ma{cc}{0 & A^* \\ A & 0}$ exists and is given by the integral
\[
\T{Ano}_m(\C D) = \frac{(-1)^n}{(2 \pi i)^n \cd n!} \cd \int \T{Tr}(\sA)
\]
of the $(2n)$-form $\sA$ defined in \eqref{eq:fordef}.
\end{theorem}
\begin{proof}
By Lemma \ref{l:onldeg} it is enough to show that
\[
\begin{split}
& (-1)^n \cd {m+n-1 \choose m-1} \cd \lim_{\la \to \infty} \big( \la^m \cd \T{Tr}_s\big( \G K^n \cd (\la + \De)^{-n-m} \big) \big) \\
& \q = \frac{(-1)^n}{ (2 \pi i )^n \cd n!} \cd \int \T{Tr}(\sA).
\end{split}
\]

Let $g \in C^\infty_c(\rr^{2n},M_N(\cc))$ denote the unique smooth compactly supported function such that $\sA = g \cd dx_1 \wlw dx_{2n}$. Let $\la > 0$. We then have that
\[
\begin{split}
\T{Tr}_s\big( \G K^n \cd (\la + \De)^{-n-m} \big) & = \T{Tr}_s\big( E_{\{1,\ldots,2n\}}(\G K^n) \cd (\la + \De)^{-n-m} \big) \\
& = \T{Tr}(g \cd (\la + \De)^{-n-m}) \cd (-2i)^n,
\end{split}
\]
since the super trace of the Clifford matrix $\pi_{2n}(e_{\{1,\ldots,2n\}}) \in M_{2^n}(\cc)$ is the number $(-2i)^n$, see Lemma \ref{l:tracli}.

Let $j \in \{1,\ldots,N\}$ and let $g_j : \rr^{2n} \to \rr$ denote the matrix entry in position $(j,j)$ of the self adjoint element $g \in M_N\big(C^\infty_c(\rr^{2n})\big)$. The bounded operator
\[
g_j \cd (\la + \De)^{-n-m} : L^2(\rr^{2n}) \to L^2(\rr^{2n})
\]
is unitarily equivalent (by the Fourier transform) to the integral operator $T_{K_j} : L^2(\rr^{2n}) \to L^2(\rr^{2n})$ with kernel $K_j \in C(\rr^{2n} \ti \rr^{2n}) \cap L^2(\rr^{2n} \ti \rr^{2n})$ given by
\[
K_j(x,y) = (2 \pi)^{-n} \cd \C F(g_j)(x-y) \cd (\la + \sum_{k = 1}^{2n} y_k^2)^{-n-m}.
\]
Here $\C F(g_j)$ denotes the Fourier transform of $g_j \in C_c^\infty(\rr^{2n})$.
 
Since $g_j \cd (\la + \De)^{-n-m}$ is of trace class it follows (see for example \cite[Corollary 10.2]{GoKr:LNO}) that
\[
\begin{split}
& \T{Tr}\big( g_j \cd (\la + \De)^{-n-m} \big) = \int_{\rr^{2n}} K_j(x,x) dx \\ 
& \q = (2 \pi)^{-2n} \cd \int_{\rr^{2n}} g_j dx \cd \int_{\rr^{2n}} (\la + \sum_{k=1}^{2n} x_k^2)^{-n-m} dx \\
& \q = (2 \pi)^{-2n} \cd \int_{\rr^{2n}} g_j dx \cd \T{Vol}(S^{2n-1}) \cd \int_0^\infty r^{2n-1} (\la + r^2)^{-n-m} dr \\
& \q = (2 \pi)^{-2n} \cd \int_{\rr^{2n}} g_j dx \cd \T{Vol}(S^{2n-1}) \cd \la^{-m} \cd \int_0^\infty r^{2n-1} (1 + r^2)^{-n-m} dr \\
& \q = (2 \pi)^{-2n} \cd \int_{\rr^{2n}} g_j dx \cd \T{Vol}(S^{2n-1}) \cd \la^{-m} \cd \frac{1}{2} \cd \frac{(n-1)! \cd (m-1)!}{(m+n-1)!},
\end{split}
\]
where $\T{Vol}(S^{2n-1}) = 2 \cd \pi^n /(n-1)!$ is the volume of the unit sphere in $\rr^{2n}$.

It follows from the above computations that
\[
\begin{split}
& (-1)^n \cd {m+n-1 \choose m-1} \cd \lim_{\la \to \infty} \big( \la^m \cd \T{Tr}_s\big( \G K^n \cd (\la + \De)^{-n-m} \big) \big) \\
& \q = (-1)^n \cd {m+n-1 \choose m-1} \cd (-2i)^n \cd \frac{1}{(4\pi)^n}  \cd \frac{(m-1)!}{(m+n-1)!} \cd \int \T{Tr}(\sA) \\
& \q = \frac{(-1)^n}{(2 \pi i)^n \cd n!} \cd \int \T{Tr}(\sA).
\end{split}
\]
This proves the theorem.
\end{proof}

It is worthwhile to notice that the anomaly of the Dirac type operator $\C D$ does \emph{not} depend on the degree $m \in \{n,n+1,\ldots\}$. This is can be seen immediately from the above theorem.

\section{Non-triviality of the anomaly \label{nont}}
The aim of this section is to demonstrate that the anomaly of the Dirac type operator $\C D : H^1(\rr^{2n}) \ot \cc^{2^n} \to L^2(\rr^{2n}) \ot \cc^{2^n}$ associated with a $U(1)$-connection can be non-trivial. We shall thus assume that the functions $a_1,\ldots,a_{2n} : \rr^{2n} \to \rr$ satisfy the conditions in Assumption \ref{a:fct}.

Recall from Theorem \ref{t:locano} that the anomaly can be computed as the integral
\[
\T{Ano}_m(\C D) = \frac{(-1)^n}{(2 \pi i)^n \cd n!} \cd \int \sA \q m \in \{n,n+1,\ldots\},
\]
where the compactly supported $(2n)$-form $\sA \in \Om^{2n}_c(\rr^{2n})$ was defined in \eqref{eq:fordef}.
\vspace{3pt}

\emph{Suppose first that $n = 1$}. The $2$-form $\sA \in \Om^2_c(\rr^2)$ is then simply the curvature form of the $U(1)$-connection used for constructing the Dirac type operator $\C D$. See the end of Section \ref{s:anodir}. This curvature form is given by
\[
d\om = -i \frac{\pa a_1}{\pa x_2} dx_1 \we dx_2
\]
and the local formula can then be rewritten as follows:
\[
\T{Ano}_m(\C D) = \frac{1}{2\pi} \int \frac{\pa a_1}{\pa x_2} dx_1 \we dx_2 \q m \in \nn
\]

Let now $\al, \be \in \rr$ and choose a smooth function $h : \rr \to \rr$ such that the derivative $\frac{dh}{dt} : \rr \to \rr$ has compact support and such that $\lim_{t \to \infty} h(t) = \be$ and $\lim_{t \to -\infty}h(t) = \al$. Choose a smooth compactly supported function $\phi : \rr \to [0,\infty)$.

Suppose then that $a_1 : \rr^2 \to \rr$ is defined by $a_1 : (x_1,x_2) \mapsto h(x_2) \cd \phi(x_1)$ for all $(x_1,x_2) \in \rr^2$. We then have that
\[
\begin{split}
\T{Ano}_m(\C D) = \frac{1}{2\pi} \cd \int_{-\infty}^\infty \frac{dh}{dt} dt  \cd \int_{-\infty}^\infty \phi(x) dx
= \frac{(\be - \al)}{2\pi} \cd \int_{-\infty}^\infty \phi(x) dx,
\end{split}
\]
for all $m \in \nn$.
\vspace{3pt}

\emph{Suppose now that $n = 2$}. In this case, the curvature form is given by $d\om = \sum_{1 \leq i < j \leq 4} g_{i,j} dx_i \we dx_j$, where the coefficient functions are defined by
\[
g_{i,j} := i \cd \Big( \frac{\pa a_j}{\pa x_i} - \frac{\pa a_i}{\pa x_j} \Big) \in C^\infty_b(\rr^{2n}) \q 1 \leq i < j \leq 4. 
\]
The compactly supported $4$-form $\sA = \si_4(\G K^2) \in \Om^4_c(\rr^4)$ can therefore be written as
\[
\sA = 2 \cd (g_{1,2} \cd g_{3,4} - g_{1,3} \cd g_{2,4} + g_{1,4} \cd g_{2,3}) dx_1 \wlw dx_4.
\]
And the local formula for the anomaly yields the identity
\[
\T{Ano}_m(\C D) = -\frac{1}{4 \pi^2} \cd \int \big( g_{1,2} \cd g_{3,4} - g_{1,3} \cd g_{2,4} + g_{1,4} \cd g_{2,3} \big) dx_1 \wlw dx_4,
\]
for all $m \in \{2,3,\ldots\}$.

Let $h : \rr \to \rr$ be as above and let $\phi_1,\phi_3 : \rr^3 \to \rr$ be smooth compactly supported functions.

Suppose that $a_4 = a_2 = 0$ and that $a_1, a_3 : \rr^4 \to \rr$ are defined by
\[
\begin{split}
& a_1 : (x_1,\ldots,x_4) \mapsto h(x_4) \cd \phi_1(x_1,x_2,x_3) \q \T{and} \\
& a_3 : (x_1,\ldots,x_4) \mapsto h(x_4) \cd \phi_3(x_1,x_2,x_3)
\end{split}
\]
for all $(x_1,\ldots,x_4) \in \rr^4$. Let $m \in \{2,3,\ldots\}$. The anomaly is then given by
\[
\begin{split}
\T{Ano}_m(\C D)
& = -\frac{1}{4 \pi^2} \cd \int h \cd \frac{dh}{dt} \, dt
\cd \int \Big( \phi_1 \cd \frac{\pa \phi_3}{\pa x_2} - \frac{\pa \phi_1}{\pa x_2} \cd \phi_3   \Big) dx_1 dx_2 dx_3 \\
& = -\frac{(\be^2 - \al^2)}{8 \pi^2} \cd \int \Big( \frac{\pa (\phi_1 \cd \phi_3)}{\pa x_2}  - 2 \cd \frac{\pa \phi_1}{\pa x_2} \cd \phi_3 \Big) dx_1 dx_2 dx_3 \\
& = \frac{(\be^2 - \al^2)}{4 \pi^2} \cd \int \frac{\pa \phi_1}{\pa x_2} \cd \phi_3 dx_1 dx_2 dx_3,
\end{split}
\]
where the integral $\int \frac{\pa (\phi_1 \phi_3)}{\pa x_2} dx_1 dx_2 dx_3$ is trivial since $\phi_1 \cd \phi_3$ has compact support.

\bibliographystyle{amsalpha-lmp}

\newcommand{\etalchar}[1]{$^{#1}$}
\def\cprime{$'$}
\providecommand{\bysame}{\leavevmode\hbox to3em{\hrulefill}\thinspace}
\providecommand{\MR}{\relax\ifhmode\unskip\space\fi MR }
\providecommand{\MRhref}[2]{%
  \href{http://www.ams.org/mathscinet-getitem?mr=#1}{#2}
}
\providecommand{\href}[2]{#2}

\end{document}